\documentclass[11pt,a4paper]{amsart}
\usepackage{a4wide}
\usepackage{amsfonts,amsthm,amsmath,amssymb}
\usepackage{amscd,color}
\usepackage{verbatim}
\usepackage{graphicx,subfigure}
\usepackage{pgfplots} 
\usepackage{mathtools}
\usepackage{marginnote}
\usepackage{soul}
\usepackage{mathtools}
\usepackage{comment}

\newtheorem*{thmA}{Theorem A}
\newtheorem*{thmB}{Theorem B}
\newtheorem*{thmC}{Theorem C}

\theoremstyle{plain}

\newtheorem{theorem}{Theorem}[section]
\newtheorem{lemma}[theorem]{Lemma}
\newtheorem{proposition}[theorem]{Proposition}

\newtheorem{remark}[theorem]{Remark}

\newcommand{\R}{\mathbb{R}}
\newcommand{\Q}{\mathbb{Q}}

\newcommand{\N}{\mathbb{N}}
\newcommand{\A}{\mathcal{A}}

\newcommand{\D}{\mathcal{D}}

\newcommand{\T}{\mathcal{T}}

\newcommand{\I}{\mathcal{I}} 
\newcommand{\HH}{\mathcal{H}}
\newcommand{\PP}{\mathcal{P}}

\newcommand{\loc}{\mathrm{loc}}

\newcommand{\wh}{\widehat}
\newcommand{\wt}{\widetilde}
\newcommand{\ol}{\overline}

\newcommand{\la}{\lambda}

\renewcommand{\r}{\color{red}}
\renewcommand{\b}{\color{blue}}

\newsavebox{\savepar}

\numberwithin{equation}{section}

\graphicspath{{./figures/}}

\begin{document}

\title[Chaotic dynamics, topological intersections and -smooth global maps] {Chaotic dynamics at the boundary of a basin of attraction via non-transversal intersections for a non-global smooth diffeomorphism}
\date{\today}

\author{Ernest Fontich}
\address{Departament de Matem\`atiques i Inform\`atica, Universitat de Barcelona (UB), 
Gran Via de les Corts Catalanes 585, 08007 Barcelona, Catalonia,  Spain
and 
Centre de Recerca Matem\`atica (CRM), 
08193 Bellaterra, Barcelona, Catalonia, Spain}
\email{fontich@ub.edu}
\author{Antonio Garijo}
\address{Departament d'Enginyeria Inform\`atica i Matem\`atiques,
Universitat Rovira i Virgili, 43007 Tarragona, Catalonia, Spain}
\email{antonio.garijo@urv.cat}
\author{Xavier Jarque}
\address{Departament de Matemàtiques i Informàtica, Universitat de Barcelona, Gran Via, 585, 08007 Barcelona, Catalonia and  Centre de Recerca Matemàtica, Edifici C, Campus Bellaterra, 08193 Bellaterra, Catalonia}
\email{xavier.jarque@ub.edu}

\thanks{The first author is supported by the
 grant PID2021-125535NB-I00.   
 The second and  third  authors were  supported by MICIU/AEI grant PID2020-118281GB-C32(33) Both grant are funded by MICIU/AEI/10.13039/501100011033/FEDER,UE.  The second author is supported by  Generalitat de Catalunya  2021SGR-633. We want to thank  the Thematic Research Programme {\it Modern holomorphic dynamics and related fields}, Excellence Initiative – Research University programme at the University of Warsaw.
Finally, this work has also been funded through the Mar\'ia de Maeztu Program for Centers and Units of Excellence in R\&D, grant CEX2020-001084-M funded by MICIU/AEI/10.13039/501100011033/FEDER,UE.}

\begin{abstract}
	In this paper we give analytic proofs of the existence of transversal homoclinic points for a family of non-globally smooth diffeomorphisms having the origin as a fixed point which come out as a truncated map governing the local dynamics near a critical period three cycle associated to the Secant map. Using Moser's version  of Birkhoff-Smale's Theorem, we prove that the boundary of the basin of attraction of the origin  contains a  Cantor-like invariant subset such that the restricted dynamics to it is conjugate to the full shift of $N$-symbols for any integer $N\ge 2$ or infinity. 
\end{abstract}

\maketitle

%
%
%
%
%
%

\section{ Introduction}

The question of whether a dynamical system admits invariant subsets of the phase portrait in which the dynamics is {\it chaotic} goes back to the origins of this area of mathematics. 
Studying the existence, or not, of chaotic dynamics and determine the topology and the geometry of the subsets where this happens  has  become a classical problem.  Nevertheless, since the question arises in so many distinct scenarios there has been different approaches to this phenomena, including the use of non-equivalent mathematical definitions in order to capture the meaning of {\it chaos} in each particular case. 
Measuring chaos in high, or even infinite, dimensional Hamiltonian dynamical systems or doing so for one-dimensional interval dynamics requires to particularize the meaning of the word {\it chaos} to concrete mathematical definitions.  

Nonetheless, once we agree on which dynamical properties characterize chaos (density of periodic points, transitivity, dense orbits, sensibility with respect to initial conditions, all at once, ...) a common accepted approach to ensure chaotic dynamics is to show that, in certain dynamically invariant region(s) of our phase portrait, the dynamics is {\it conjugate} (that is, {\it equal} up to a homeomorphism) to the one of a model for which it is somehow easy to test the properties mentioned above. 

The usual toy model is the dynamical system $(\Sigma_N,\sigma)$, where 
$\Sigma_N$ is the set of bi-infinite (or one-side) sequences of $N\ge 2$ symbols and $\sigma$ is the {\it shift map}; see \cite{Mos}. One can easily check that the system $(\Sigma_N,\sigma)$ captures  the dynamical properties proposed above.  Since the conjugacy sends orbits of our dynamical systems to orbits of the shift map acting on the space of symbols, this methodology is also known as {\it symbolic dynamics}. To focus on the content of this paper and simplify the discussion, let us assume we have a discrete dynamical system in $\mathbb R^2$ generated by the iterates of a (smooth) map.

In any event, the difficult part to apply this strategy is to show that in some regions of the phase portrait our dynamics is conjugate to the dynamical system $(\Sigma_N,\sigma)$. A major result in this direction goes back to the cornerstone ideas of S. Smale (Birkhoff-Smale's Theorem) and J. Moser \cite{Mos} who provide {\it checkable} (in some cases only numerically) dynamical conditions to ensure that a given dynamical system has a subset of the phase portrait whose dynamics is conjugated to the full-shift of an arbitrary number of symbols (even infinitely many). Roughly speaking they showed that if a smooth map has a transversal homoclinic intersection between the stable and unstable invariant manifolds of a hyperbolic saddle fixed point then, there is an invariant Cantor set whose restricted dynamics is conjugate to $(\Sigma_N,\sigma)$.

Even though the results have been extremely helpful in many different contexts (and extended in many different directions) we emphasize that the hypotheses include three key ingredients: the {\it hyperbolicity} of the saddle point, the map is a {\it global diffeomorphism} and the {\it transversality} of the intersection of the invariant manifolds. The main goal of this paper is to address the presence of chaotic dynamics, for a concrete family of maps, under the lack of two of the conditions; the inverse map would not be globally smooth and in a first step we only can prove (analytically) that we have an intersection with a finite order contact.

Concretely, in this paper we consider the map
\begin{equation}\label{eq:T}
	T_{d}\left(
	\begin{array}{l}
		x   \\ 
		y 
	\end{array}
	\right)  = 
	\left(
	\begin{array}{l}
		y -  (x+y)^d \\ 
		y - 2(x+y)^d
	\end{array}
	\right), 
\end{equation}
with $d\geq 3$ being odd. Such map is a {\it truncated} expression  of the third iterate of the (extended) Secant map applied to a polynomial $p(x)$ near a critical period thee cycle 
$$
(c,c) \mapsto (c,\infty) \mapsto (\infty,0) \mapsto (c,c),
$$  
where $p^{\prime}(c)=0$ (but $p(c)\ne 0$). See \cite{BedFri, Tangent,GarJar_infinity, FGJ24} for more details. For later discussions we point out here that $T_d$ is a global homeomorphism, but it is not a global diffeomorphism since the inverse map, $T_d^{-1}$, is not smooth over the line $\{y=x\}$ (see \eqref{eq:T_inverse_odd} for its particular expression). 

One can  easily check that the origin of \eqref{eq:T} is a fixed point and its basin of attraction 
\begin{equation}\label{eq:basin}
	\mathcal A_{d}(0) = \{ (x,y) \in \mathbb R^2\, | \ T_{d}^n(x,y) \to (0,0) \,  \hbox{ as }  \, n \to \infty  \} 
\end{equation}
is not empty. In \cite{FGJ24} we proved the following topological description of 
$\mathcal A_{d}(0)$ and further information about its boundary. We denote  $p_0=(0,1)$ and $p_1=(0,-1)$.

\begin{theorem}\label{theorem:paper1}
	Let $d\geq 3$ odd. Then $\mathcal A_{d}(0)$ is an open, simply connected, unbounded  set. Moreover, $\partial  \mathcal A_{d}(0)$ contains the stable manifold of the hyperbolic two-cycle  $\{p_0,p_1\}$ lying in $\partial  \mathcal A_{d}(0)$. 
\end{theorem}

\begin{figure}[ht]
	\centering
	\includegraphics[width=0.3\textwidth]{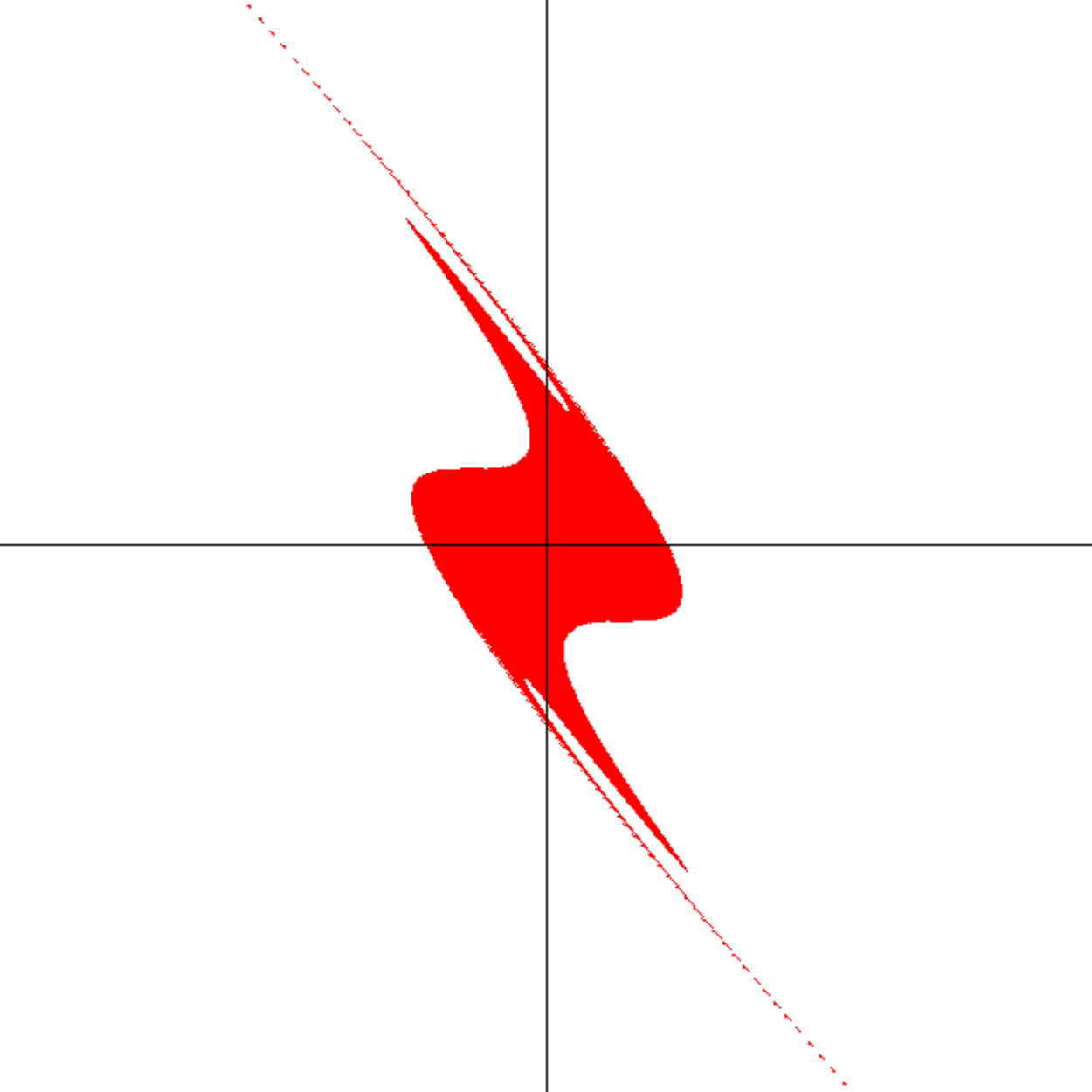}
	\put(-69,136){\tiny $y$}
	\put(2,65){\tiny $x$}
	\put(-58,78){\tiny $\mathcal A_{3}(0)$}
	\caption{\small{The picture (in red) of the set $\mathcal A_{3}(0)$. Notice that according to Theorem \ref{theorem:paper1} the red region is connected, simply connected and unbounded}}
	\label{fig:dynamical_plane}
\end{figure}

The thesis of the above theorem glimpse the possible topological complexity of $\partial \mathcal A_{d}(0)$ (see Figure \ref {fig:dynamical_plane}). In fact, the main goal of this paper is to provide a better understanding of $\partial \mathcal A_{d}(0)$ by proving that, apart from the stable manifold of the hyperbolic two-cycle  $\{p_0,p_1\}$ there is a Cantor subset of $\partial \mathcal A_{d}(0)$ where the dynamics is conjugated to the
one of the shift of $N$ symbols and so inhering all its chaotic dynamics.

In \cite{FGJ24} we were able to describe and bound the shape of a piece of the unstable manifold of $p_0$ for $T^2_d$  (and it was a key point in the arguments to prove Theorem \ref{theorem:paper1}). In this paper we mimic some of the arguments there to control the shape of a piece of the stable manifold of $p_1$.  Using both constructions and a singular $\lambda$-Lemma \cite{Ray} we can ensure the existence of homoclinic (not necessarily linearly transversal) points for $T_d^2$.

\begin{thmA} 
	Let $\{p_0,p_1\}$ be the hyperbolic two cycle lying in the boundary of $\partial  \mathcal A_{d}(0)$. Then, the stable and unstable manifolds of $p_1$ (as well as $p_0$), as a fixed point for $T_d^2$,  intersect at a homoclinic point. 
\end{thmA}

Going back to our previous arguments, if we want to apply Birkhoff-Smale's Theorem we need to prove the existence of transversal homoclinic points, so that Theorem A is not enough. In \cite{CR} the authors are able to conclude transversal intersections under the presence of (topological) homoclinic intersections, but their map is area preserving, is a global smooth diffeomorphism and admits, in a sufficiently small neighbourhood the hyperbolic saddle, a concrete local normal form which provides a first integral. 

In our case we have not the previously mentioned  normal form and since $T_d^{-1}$ is not smooth over the line $\{y=x\}$ we cannot use that the globalization of the stable manifold of the 2-cycle $\{p_0,p_1\}$ by applying $T_d^{-2}$ is analytic. In any event, inspired in the strategy proof in \cite{CR}, using alternative  arguments to deal with our weaker conditions we are able to conclude the existence of transversal homoclinic intersections. 

\begin{thmB}
	Let $\{p_0,p_1\}$ be the hyperbolic two cycle lying in the boundary of $\partial  \mathcal A_{d}(0)$. Then, the stable and unstable manifolds of $p_1$ (as well as $p_0)$, as a fixed point for $T_d^2$, intersect transversally. 
\end{thmB}

Although from the previous theorem, we have the existence of transversal homoclinic points we still cannot directly apply Birkhoff-Smale's Theorem since the inverse map, $T_d^{-1}$ map is not globally smooth. However, we can overcome this difficulty and prove the main result of this paper.

\begin{thmC}
	There exists an invariant Cantor set, contained in  $\partial \A_{a,d}(0)$, where the dynamics of $T_d^2$ is conjugate to the full shift of $N$-symbols. In particular,  $\partial \A_{a,d}(0)$ contains  infinitely many periodic points with arbitrary high period.
\end{thmC}

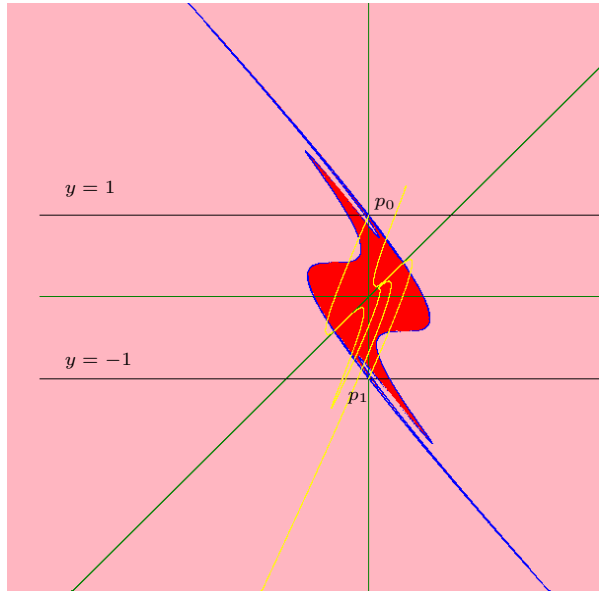
\begin{figure}[ht]
	\centering
	\includegraphics[width=0.5\textwidth]{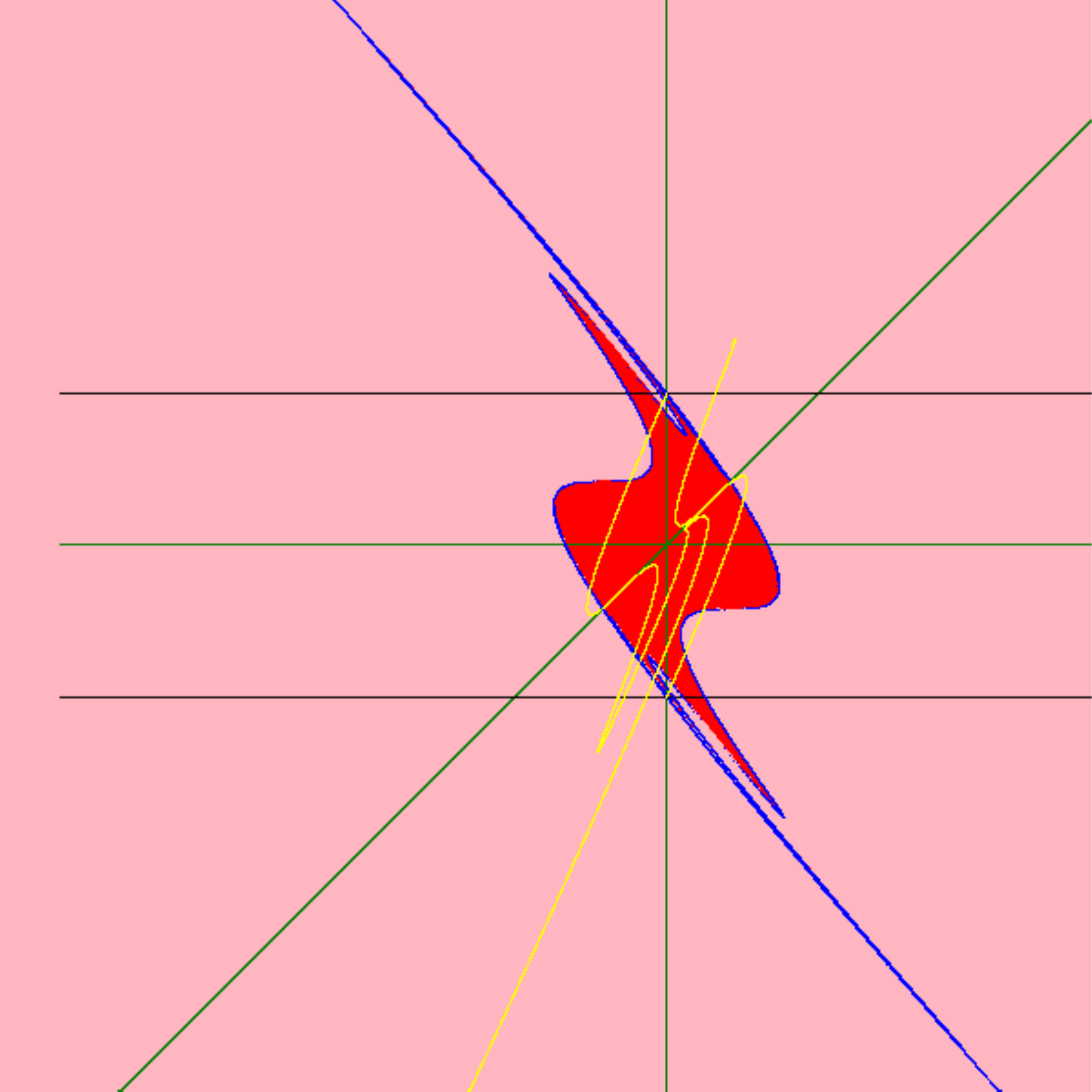}
	\put(-200,150){\tiny $y=1$}
	\put(-200,85){\tiny $y=-1$}
	\put(-84,145){\tiny $p_0$}
	\put(-94,72){\tiny $p_1$}
	\caption{\small{This picture corresponds to $d=3$. In red we plot the attracting basin $\mathcal A_3(0)$. In blue (respectively, yellow) we draw the stable (respectively, unstable) manifold of the two cycle $\{p_0,p_1\}$. The picture illustrate (numerically) the transversal intersections described in Theorem B. According to Theorem C, $\partial \A_{d}(0)$ contains the stable manifold of the two cycle (in blue) and a Cantor-set like with chaotic dynamics.}}       
	\label{fig:intersection_global}
\end{figure}

We emphasize that the theorems provide analytic proofs, rather than numerical  evidence, of  non-local properties of invariant manifolds for a family of maps. There are few cases where this has been done. For instance, in \cite{Fon90}, there is an analytical proof of the transversal intersection of the invariant manifolds for a wide range of a parameter  for a class of maps which include the conservative H\'enon map and the Chirikov standard map.
In \cite{Gel99} there is an analytical proof of the transversal intersection of the invariant manifolds of the standard map when the angle is exponentially small with respect to the parameter of the family. Also, in \cite{DM96} and \cite{MSS11} they prove transversal intersection for the manifolds of close to integrable maps.

We organize the paper as follows. In Section \ref{sect:prelimpaper1} we summarize some preliminaries from \cite{FGJ24} that we need in the proofs of the present paper, trying to make the present paper self-contained.  In Section \ref{sec:homoclinic} we prove Theorem A, in Section \ref{section:proof_B} we prove Theorem B and finally in Section \ref{section:proof_C} we conclude the proof Theorem C.
\section{Preliminaries} \label{sect:prelimpaper1}
	
In this section we collect some preliminary results about the map $T_d:\R^2\to \R^2$, introduced in (\ref{eq:T}), for $d\geq 3$, an odd number. Everything was already introduced in \cite{FGJ24} but, for the sake of completeness and easier reading, we include them here. 

The map $T_d:\R^2\to \R^2$ is a polynomial and a homeomorphism and its inverse map is real analytic in $\R^2 \setminus \{x=y\} $, but not differentiable on the line $\{x=y\}$. Its inverse is  given by
\begin{equation}\label{eq:T_inverse_odd}
		T^{-1}_{d} (x,y) = \left( -2x+y + \left( x-y\right)^{1/d} , \, 2x-y \right).
	\end{equation}
Observe that $T_d^{-1}(x,x)=(-x,x)$  for all $x \in \mathbb R$. One can easily check that $T_d$ has a unique two-cycle  $\{p_0=(0,1), p_1=(0,-1)\}$,  i.e.,  $p_1= T_d(p_0)$ and  $p_0= T_d(p_1)$. This two-cycle will play a fundamental role in the dynamics of $T_d$.  Moreover, we have that 
\begin{equation} \label{eq:dif_matrix_two_cycle}
	DT_d^2(p_0)=DT_d^2(p_1)= DT_d(p_0) DT_d(p_1)
	=
	\left (
	\begin{array}{ll}
		3d^2-2d & \ \ 3d^2-4d+1 \\
		6d^2-2d & \ \ 6d^2-6d+1 
	\end{array}
	\right ).
\end{equation}
A direct computation shows that the 
characteristic equation of $DT_d^2(p_0)$ is 
$$
p(\la)= \la^2 -(1- 8d  + 9 d^2)\la  + d^2 =0
$$
and the
eigenvalues and eigenvectors are given by 
\begin{equation}\label{eq:lambda}
	\lambda^{\pm}_d =  \frac{1}{2}\left ( 9d^2- 8d  + 1 \pm (3d-1)\sqrt{ 9 d^2 -10d + 1} \right) 
\end{equation}
and
\begin{equation}\label{eq:m}
	 (1,m^{\pm}_d) =   \left(1,\ \frac{ 4d}{1-d \pm \sqrt{ 9d^2-10d +  1}} \right),
\end{equation}
respectively. 

On the one hand, it is easy to check that both eigenvalues are strictly positive. Moreover,  
$\lambda^{-}_d$ is strictly decreasing and  $\lambda^+_d$ is strictly increasing, with respect to the parameter $d$. We also have
\begin{center}
\begin{tabular}{ccl}
$ \lim\limits _{d\to \infty}  \lambda^{-}_d=1/9$ &\quad  and &\quad  $1/9<\lambda^{-}_d \le   
\lambda^{-}_3 =  29-8\sqrt{13} \approx 0.1556$ \\
$ \lim\limits_{d\to \infty}  \lambda^{+}_d=\infty$ &\quad  and & \quad 
 $\lambda^{+}_ d \ge \lambda^{+}_ 3 = 29+8\sqrt{13}\approx  57.8444.$
\end{tabular}
\end{center}

On the other hand, $m^{-}_d$  is negative and  strictly  increasing while  $m^+_d$ is positive and  strictly decreasing (both with  respect to the parameter $d$). We also have

\begin{center}
\begin{tabular}{ccl}
$\lim\limits_{d\to \infty} m^{-}_d =-1$ &\quad  and &\quad  
 $-1.3028 \approx \frac{-6}{1+\sqrt{13}} = m^{-}_3  \le m^{-}_d <-1$ \\
 $ \lim\limits _{d\to \infty} m^{+}_d =2$ &\quad  and &\quad  $2<m^{+}_d \le m^{+}_3 = \frac{6}{\sqrt{13}-1} \approx 2.3028$ 
 \end{tabular}
\end{center}

Therefore,  the two cycle  $\{p_0, p_1\}$ is  hyperbolic of  saddle type. 
We denote $W^s_{p_j}$, $ W^u_{p_j}$ the stable and the unstable manifolds of the fixed points $p_j$ for the map $T_d^2$, $j=0,1$. Similarly we denote by $W^s_{\loc,\, p_j}$, $ W^u_{\loc,\, p_j}$ the corresponding local stable and  unstable manifolds of some size $\delta$ that we do not make explicit in the notation. Actually, given some size $\delta>0$, 
$$
W^s_{\loc,\, p_j} = \{ z\in \R^2\mid \, T^{2k}_d (z) \in B_\delta (p_j) \ \ \text{ for all} \ \ k\ge0 \},
$$ 
where $B_{\delta}(p_j)$ denotes the open ball centered at $p_j$ with radius $\delta>0$, for $j=0,1$. We define  analogously  $W^u_{\loc,\, p_j}$ for $T^{-2k}_d$. 

We also denote 
$$
W^s:=W^s_{\{p_0,p_1\}} \quad \mbox{\rm and} \quad W^u:=W^u_{\{p_0,p_1\}}$$ 
the global stable and unstable manifolds of the periodic orbit $\{p_0,p_1\}$, respectively. 
Since $T_d$ is analytic on $\mathbb R^2$ and $T^{-1}_d$ is analytic on $\mathbb R^2 \setminus  \{y=x\}$ the local versions of the invariant manifolds are analytic. Moreover, the (global) unstable  manifold, obtained iterating by $T_d$ the local one, is analytic and the (global) stable manifold, obtained iterating by $T^{-1}_d$, is analytic except at the preimages of the intersections of $W^s$ with $\{y=x\}$.

When there is no confusion we use the simplified notation $\lambda^{\pm} := \lambda^{\pm}_d$ and $m^{\pm} := m^{\pm}_d$.

\subsection*{The triangle $\mathcal D$ and its images: $T_d(\D)$ and $T_d^{-1}(\D)$.}

In \cite[Section 5]{FGJ24} we considered the triangle $\D$ of vertices 
	$$
	p_1=(0,-1), \qquad \left(\frac{1}{ m^+ +1},\frac{-1}{ m^+ +1}\right) \qquad \text {and} \qquad \left(\frac{1}{m^\star+1},\frac{-1}{m^\star+1}\right), 
	$$
	where $m^\star = 7/2$, or equivalently,  
	\[
	\D = \{ (t,-1+mt) \, | \     t \in [0,1/(m+1)], \ m \in [m^+, m^\star]   \}. 
	\]
	We also considered the sets $T_d(\D)$ and $T^{-1}_d(\D)$. 
	We showed that the set $T_d(\D)$ is bounded by the images of the sides of $\D$ given by the curves 
	$\gamma_{m^+}(t)$, $\gamma_{ m^\star}(t)$ where 
	\begin{equation} \label{gamma-m}
		\gamma_{m}(t)
		=  T_d (t,-1+m t) =  
		( mt -1 - ( (m+1)t -1)^d , mt -1 - 2( (m+1)t -1)^d),
	\end{equation}
	for $ 0\le t\le \frac{1}{m+1}$, and 
$$
\partial T_d(\D) \cap \{y=x\}=\Big\{(t,t) \mid \ \frac{-1}{ m^++1} \le  t \le \frac{-1}{m^\star+1}\Big\}.
$$ 
	
Finally, we claim that there is a (connected) piece of $W^u_{p_0} \cap \{y\le 1\}$, tangent to the line $y=1+m^{+}x$ at $p_0$, contained in $T_d(\D)$ joining the point $p_0$ with some point in $\partial T_d(\D) \cap \{y=x\}$. We call left and right boundaries of  $T_d(\D)$ the curves $\gamma_{m^+}(t)$ and $\gamma_{m^\star}(t)$, respectively.  See Figure \ref{fig:D_whD} (left). We do not include here the arguments used in \cite[Lemma 5.4]{FGJ24} to prove the claim but in the next section we mimic, including all computations, the  ideas used in \cite{FGJ24} for the case of  $\wh\D$, $T_d(\wh \D)$ and $T_d^{-1}(\wh\D)$.

\section{Proof of Theorem A}\label{sec:homoclinic}

To prove Theorem A we first show the existence of an heteroclinic intersection for the map $T_d$. More precisely, we have the following statement.

\begin{proposition} \label{prop:hetero}
Let $\{p_0,p_1\}$ be the hyperbolic two-cycle lying in the boundary of $\partial  \mathcal A_{d}(0)$ (see Theorem \ref{theorem:paper1}). Then,  the unstable manifold of $p_0$ and the stable manifold of $p_1$ intersect in a heteroclinic point.
\end{proposition}

 The idea is to show that $T_d(\D)$ (Figure \ref{fig:D_whD}, left) and $T_d^{-1}(\wh\D)$ ((Figure \ref{fig:D_whD}, right) intersect in a suitable manner that forces the intersection of the invariant manifolds  (Figure \ref{fig:unbounded_global_2}).
Since the proof of this proposition is quite long, we split it into several lemmas.  We assume all notation introduced in Section \ref{sect:prelimpaper1}. In particular we have described the construction provided in  \cite[Lemma 5.4]{FGJ24} to localize the piece of the unstable manifold attached to $p_0$ inside $T_d({\D})$. The first step is to make
 a similar construction to localize a piece of the stable manifold of $p_1$. Let  
\begin{equation}\label{eq:mstar}
-\frac{3}{2}\le \wh m^{\star}=	\wh m_d^{\star}:=-1-\frac{1}{d-1} = \frac{-d}{d-1} < m^{-}<-1, 
\end{equation}
where the inequalities follow from direct computations. We introduce the triangle $\wh{\mathcal D}$ with  vertices
$$
p_0=(0,1), \qquad \left(\frac{-1}{m^- +1},\frac{-1}{m^- +1}\right) \qquad \text {and} \qquad \left(\frac{-1}{\wh m^\star+1},\frac{-1}{\wh m^\star+1}\right),
$$
or equivalently,
\[
\wh \D = \{ (t,1+mt) \, | \ t \in [0,1/(1-m)], \ m \in [\wh{m}^\star, {\b m^-}]\}. 
\]

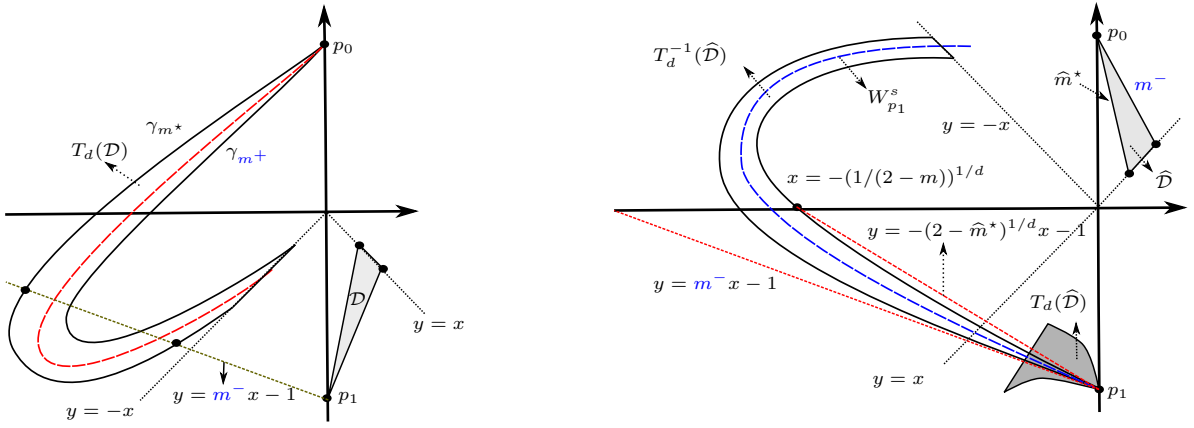
\begin{figure}[ht]
    \centering
     \includegraphics[width=\textwidth]{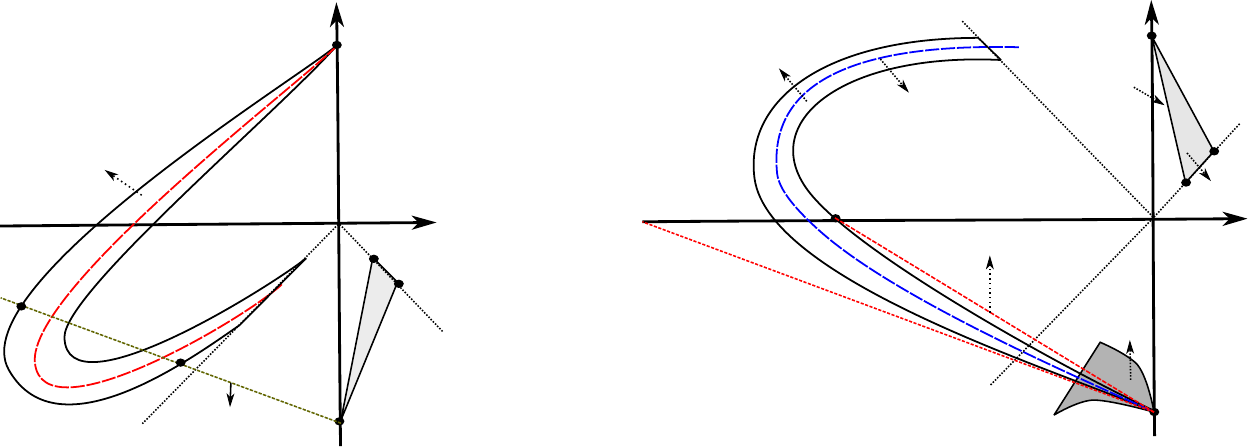}
     \put(-418,100){\tiny $T_d(\D)$}
     \put(-314,44){\tiny $\D$}
     \put(-118,14){\tiny $y=x$}
     \put(-290,35){\tiny $y=x$}
     \put(-92,110){\tiny $y=-x$}
     \put(-420,2){\tiny $y=-x$}
      \put(-390,110){\tiny $\gamma_{m^\star}$}
       \put(-360,100){\tiny $\gamma_{{\b m^{+}}}$}
     \put(-30,145){\tiny $p_0$}
     \put(-30,10){\tiny $p_1$}
     \put(-320,140){\tiny $p_0$}
     \put(-318,8){\tiny $p_1$}
      \put(-20,125){\tiny ${\b m^-}$}
      \put(-50,127){\tiny $\wh{m}^{\star}$}           
      \put(-12,88){\tiny $\wh\D$}
      \put(-200,50){\tiny $y={\b m^-}x-1$}
      \put(-380,8){\tiny $y={\b m^-}x-1$}
      \put(-120,70){\tiny $y=-(2-\wh m^\star)^{1/d}x-1$}
        \put(-150,90){\tiny $x=-(1/(2-m))^{1/d}$}                        
        \put(-59,43){\tiny $T_d(\wh{\D})$}
       \put(-200,136){\tiny $T_d^{-1}(\wh{\D})$}
       \put(-120,120){\tiny $W_{p_1}^{s}$}
           \caption{\small{Left: The triangle $\D$, its image $T_d(\D)$ and (dashed, red) a piece of $W_{p_0}^u$ attached to $p_0$. Right:    The triangle $\wh \D$, its images $T_d(\D)$ and $T_d^{-1}(\wh \D)$, and (dashed, blue) a piece of $W_{p_1}^s$ attached to $p_1$. We also add the relevant objects appearing in the proof of Proposition \ref{prop:hetero} and Theorem A.}}       
    \label{fig:D_whD}
    \end{figure}

As we did with the set $\D$ in  \cite[Lemma 5.3]{FGJ24}, we study the geometry of the sets $T_d(\wh{\mathcal{D}})$ and $T_d^{-1}(\wh {\mathcal D})$. From the properties of these sets,  we will prove that there is a piece of  $W^s_{p_1}\cap \{ y\ge -1\}$ that is contained in $T_d^{-1}(\wh {\mathcal D})$. Moreover, this piece joints $p_1$ with a point in $T_d^{-1}(\wh {\mathcal D}) \cap \{y=-x\}$. See the right picture in Figure \ref{fig:D_whD}. Then, using the geometry of the intersection of  $T_d(\D) $ and $T_d^{-1}(\wh {\mathcal D})$ we will prove that $W^u_{p_0}$ and $W^s_{p_1}$ have to cross (topologically) in a heteroclinic intersection, proving Proposition \ref{prop:hetero}. From this heteroclinic intersection we will obtain a homoclinic intersection as claimed in Theorem A.

\vglue 0.1truecm

\noindent {\bf The preimage  $T_d^{-1}(\wh {\mathcal{D}})$.}  We denote by $\wh{\Gamma}_m(t)$  the image by $T_d^{-1}$ 
of the segment $\{(t,1+mt)\mid t\in [0,1/(1-m)] \}$.  Thus,
\begin{equation}\label{defGammatilde}
\wh{\Gamma}_m(t) =   T_d^{-1}(t,1+mt)= : (\wh  {\alpha}_m(t),\wh  {\beta}_m(t)),
	\end{equation}
where 
\[
\wh  {\alpha}_m(t) = (m-2)t +1 + ( (1-m)t -1)^{1/d}  \quad \mbox{\rm and} \quad  \wh  {\beta}_m(t) = (2-m)t -1.
\]
We are interested in $\wh{\Gamma}_m(t)$ for $m \in  [\wh m^\star,m^-]$. Note that the point on $\wh \D \cap \{y=x\}$ corresponds to $t= 1/(1-m)$ and is mapped by $T_d^{-1} $ to $(-1/(1-m), 1/(1-m))$ on the line $\{y=-x\}$. Taking derivatives we have that
\[
\wh  {\alpha}_m'(t) = m-2 + \frac{1-m}{d} ((1-m)t -1)^{(1-d)/d}, \ \  \wh{\beta}_m'(t) = 2-m >0, \ \ \wh  {\alpha}_m''(t)>0 \ \ \mbox{\rm and} \ \ \wh  {\beta}_m''(t)=0.
\]
A direct computation shows that $\wh  {\alpha}'_m(t)=0$ if and only if $t= t_{\pm}$, where
\[
t_{\pm} = \frac{1}{1-m}\left(1 \pm \left( \frac{1-m}{d(2-m)}\right)^{d/(d-1)}\right)
\]
and 
$0<t_- < \frac{1}{1-m} < t_+$, since, as $m<0$, we have
$0<\frac{1-m}{d(2-m)} < 1$. 
It follows from these computations that $ \wh  {\alpha}_m(t),\ m \in  [\wh m^\star,m^-],$ has a unique minimum (in its domain)  at $t_- \in(0, \frac{1}{1-m})$. Finally,  $\wh  {\alpha}'_m( \frac{1}{1-m})=\infty$
which means that when $\wh{\Gamma}_m(t)$ meets $\{y=-x\}$, its tangent line is horizontal. See the right picture in Figure \ref{fig:D_whD}. 
In other words the vectors $\wh{\Gamma}_m'( \frac{1}{1-m})$ and $\wh{\Gamma}_m'(t_-)$ are parallel to the lines $y=0$ and $x=0$, respectively.

Since $\wh {\beta}_m(t)$ is invertible (linear), for any $m$ we can represent the curve $\wh{\Gamma}_m(t)$ as the graph of a function $x=g(y),\ y\in[-1,1/(1-m)]$ (remember that $1/(1-m)>0$), by taking
$ g(y) = \wh {\alpha}_m \circ \wh {\beta}^{-1}_m(y)$. 
Since $\wh  {\beta}_m''(t)=0 $, we have that 
$$
\frac{dg}{dy}(y)= \left[\frac{d\wh {\alpha}_m}{dt} \left(\frac{d\wh {\beta}_m}{dt}\right)^{-1}\right] \circ \wh {\beta}^{-1}_m(y) \qquad \mbox{\rm and} \qquad 
\frac{d^2g}{dy^2}(y)= \frac{d^2 \wh {\alpha}_m}{dt^2} \left(\frac{d\wh {\beta}_m}{dt}\right)^{-2} \circ \wh {\beta}^{-1}_m(y)>0.
$$
The convexity of $g$  implies that the image of $\wh{\Gamma}_m(t)$ is above its tangent line at $p_1= (0,-1)$. In case $m=  m^-$, this tangent line has slope $m^-$ and it is the minimum slope for all $m\in[\wh{m}^{\star},{m}^-]$. Therefore,   $T_d^{-1}(\wh {\mathcal{D}})$ is above the line $y=m^- x-1$.

Also,  $g$ has a unique minimum at $y_-=\wh \beta_m(t_-)$. Moreover, $\wh{\Gamma}_m(t)$ intersects $\{y=0\}$ when $t= 1/(2-m)$ at the point 
$(x,y)=(\wh {\alpha}_m (1/(2-m))  ,0) $ with 
$$
\wh {\alpha}_m (1/(2-m)) = -\left(\frac{1}{2-m}\right)^{1/d}.
$$
Again, the convexity of the function $g$ implies that  its graph intersected with $\{y\le 0\}$ is below the line
$$
y =- (2-m)^{1/d} x -1
$$ 
and, in particular (see Figure \ref{fig:D_whD}), taking $m=m^\star$ we conclude that $T_d^{-1}(\wh \D) \cap \{y\le 0\}$ is below 
\begin{equation}\label{eq:tangentbelow}
y = -(2-\wh m^\star )^{1/d} x-1=-\left(\frac{3d-2}{d-1}\right)^{1/d}x-1.
\end{equation}

\vglue 0.1truecm

\noindent {\bf The image  $T_d(\wh {\mathcal D})$.} We notice that since $\{p_0,p_1\}$ is a two-cycle we have 
$T_d(p_0)=T_d^{-1}(p_0)=p_1,$
so that  $T_d(\wh {\mathcal D})$ is attached to $p_1$ as it was the case of $T_d^{-1}(\wh {\mathcal D})$.

We denote by $\wh  {\gamma}_m(t)$ the image by $T_d$ of the segment
$\{(t,1+mt)\mid t\in [0,1/(1-m)] \}$,  with $m \in  [\wh m^\star ,m^-] $. Hence, $\wh  {\gamma}_m(t) = T_d(t,1+mt) =:(\wh  {x}_m(t),\wh  {y}_m(t))$ where 
\begin{equation}\label{eq:xmym}
\wh  {x}_m(t) = mt +1 - ( (m+1)t +1)^d , \qquad  \wh  {y}_m(t) = mt +1 - 2( (m+1)t +1)^d.
\end{equation}
To simplify notation, we write $x(t):=  \wh x_m(t)$ and $y(t):=  \wh y_m(t)$ and $\gamma(t)=\wh \gamma_m(t)$ unless it is strictly necessary to show the dependence in $m$. The derivatives are given by
\vglue0.2truecm
\begin{tabular}{ll}
$ x^\prime(t) = m  - d(m+1)( (m+1)t +1)^{d-1}$, & $ y^\prime(t) = m - 2  d(m+1)( (m+1)t +1)^{d-1}$, \\
$x^{\prime\prime}(t) =  - d(d-1)(m+1)^2( (m+1)t +1)^{d-2}$,  &  $y ^{\prime\prime}(t) =  - 2  d(d-1)(m+1)^2( (m+1)t +1)^{d-2}$.
\end{tabular}
\vglue0.2truecm

Since $t\in  [0,1/(1-m)] $ and $m<-1$ we have the inequalities
$$
0 <  \frac{2}{1-m} = \frac{m+1}{1-m} + 1 <(m+1)t +1 \leq 1.
$$
Then, for $d\geq 3$ (odd), we have $${x}''(t)<0 \qquad \text{and} \qquad  {y}''(t) <0.$$ Next lemma provides basic estimates on the parametrization  $\gamma(t)$.


\begin{lemma} \label{fitescomponentsgamma}
	Let $m\in [\wh m^\star, m^-] $ and $t\in[0,1/(1-m)]$. The following conditions hold.
	\begin{itemize}
		\item[(a)] $x(0)=0$,  $x(\frac{1}{1-m}) <  \frac{1}{1-m}[1 -\frac{2}{\sqrt{e}}] <0$, $x(t)<0$ for $t\ne 0$, and $y(t)<0$.
		\item[(b)]  $x' (t) \leq 0$ with $x^{\prime}(t)=0$ if and only if $t=0$ and $m=\wh{m}^{\star}$.
		\item[(c)]  $y'(t) >0$ for $m=m^{\star}$.
	\end{itemize}
\end{lemma}

\begin{proof}
The proof of the items follows from some computations based on the expressions of $x(t)$, $y(t)$ and their derivatives above.

Easily ${\b x}(0)=0$. On the one hand, we have
\begin{equation*}
\begin{split}
 x\Big(\frac{1}{1-m}\Big) = \frac{1}{1-m} - \Big(\frac{2}{1-m}\Big)^d 
	&= \frac{1}{1-m}\Big[1-2\Big[\Big(1+ \frac{1}{2(d-1)}\Big)^{2(d-1)}\Big]^{-1/2}  \Big] \\
&<  \frac{1}{1-m}\Big[1 -\frac{2}{\sqrt{e}}\Big] <0.
\end{split}
\end{equation*}

On the other hand, $x^{\prime}(0) = m-d(m+1)\le (-1-\frac{1}{d-1} ) (1-d)  -d =0$ where the equality only holds for $m=\wh{m}^{\star}$ and $x^{\prime\prime}(t)<0$ (see \eqref{eq:mstar}). Hence $x^{\prime}(t) <0$ (unless  $t=0$ and $m=\wh{m}^{\star}$ where $x_{m^\star}^{\prime}(0)=0$) and so $x(t)$ is decreasing (and negative unless $t=0$). Finally, we have $y(t) = x(t) - ( (m+1)t +1)^d <0$. All together implies (a) and (b).

If $m=\wh m^\star$, using \eqref{eq:mstar} we have
\begin{align*}
 {y}'_{m^{\star}}(t) &  = m^{\star} - 2  d(m^{\star}+1)( (m^{\star}+1)t +1)^{d-1} \ge \frac{d}{d-1} \Big[-1+ 2\Big(\frac{-1}{d-1} t + 1\Big)^{d-1}\Big]		\\
 & \ge \frac{d}{d-1} \Big[-1+ 2\Big(1 - \frac{1}{2d-1}\Big)^{d-1}\Big] = \frac{d}{d-1} \Big[-1+ 2\Big(1+ \frac{1}{2(d-1)}\Big)^{1-d}\Big]
 \\ & >   \frac{d}{d-1} \Big[-1+ \frac{2}{\sqrt{e}}\Big] >0.
\end{align*}
that proves  (c).
\end{proof}

Since $x'(t)<0$, the function $x(t)$ is invertible. 
If $t=t(x)$ is the inverse map of $x(t)$, then (the image of) $\gamma(t)$ can be represented as the graph of the function $h(x):=h_m(x):= y\circ t(x)$. From its definition the function $h$ is smooth.
\begin{lemma}
We have that $h(x)=y \circ t (x)$ is concave.
\end{lemma}
\begin{proof} Taking derivatives we have  
\[
h^{\prime}(x) =  \frac{y^{\prime}}{x^{\prime}}\circ t(x)\qquad \text {and} \qquad 	h^{\prime\prime}(x) 
=\frac{1}{(x^{\prime})^2}[y^{\prime\prime}  -  x^{\prime\prime} \frac{y^{\prime}}{x^{\prime}} ]\circ t(x). 
\]

From the expressions of $x$ and $y$  and their derivatives (see \eqref{eq:xmym} and the derivatives below) we have ${y}^{\prime}(t) =   {x}^{\prime}(t)- d(m+1)( (m+1)t +1)^{d-1}$ and $  {y}^{\prime\prime}(t) = 2   {x}^{\prime\prime}(t) $. Therefore
\[
h^{\prime\prime}(x) =\frac{x^{\prime\prime}}{(x^{\prime})^2}[2  -  \frac{y^{\prime}}{x^{\prime}} ]\circ t(x) \qquad \text {and} \qquad 
2-\frac{y'}{x'} = 1+ \frac{d(m+1)( (m+1)t +1)^{d-1}}{x'} >1,
\]
concluding $h^{\prime\prime}(x) <0$ and hence $h$ is concave (remember that the result is valid for all values of $m$ in the range). 	
\end{proof}

Now we fix $m=\wh{m}^\star$. We claim that (the image of) $\wh \gamma _{\wh m^\star}$ belongs to $\{x\le 0,y\le 0 \}$ and it is above $T^{-1}_d(\wh \D) $. To check the claim, accordingly to the previous study of $T^{-1}_d(\wh \D)$  it is sufficient to check that $\wh \gamma _{\wh m^\star}$ is above the line $y= -(2-\wh m^\star)^{1/d} x-1$ introduced in \eqref{eq:tangentbelow}. Moreover, since (the image of)
$\wh \gamma _{\wh m^\star}$ is the graph of a concave function it is enough to check that
$$
y_{\wh{m}^{\star}}\left(\frac{1}{1-\wh{m}^{\star}}\right) > -(2-\wh{m}^{\star})^{1/d} \ x_{\wh{m}^{\star}}\left(\frac{1}{1-\wh{m}^{\star}}\right)-1.
$$
This is equivalent to 
\[ 
\frac{1}{1-\wh m^\star} - 2 \Big(\frac{2}{1-\wh m^\star}\Big)^d +  (2-\wh m^\star)^{1/d} \Big[  \frac{1}{1-\wh m^\star} - \Big(\frac{2}{1-\wh m^\star}\Big)^d \Big] +1>0.
\] 
If we substitute $\wh m^\star = -1-\frac{1}{d-1}$, the above inequality writes as 
\begin{equation} \label{desigualtat-perawhD}
	1 -  \Big(\frac{2}{2+1/(d-1)}\Big)^d + \Big[\frac{1}{2+1/(d-1)} - \Big(\frac{2}{2+1/(d-1)}\Big) ^d \Big]\Big[1+\Big(3+ \frac{1}{d-1}\Big)^{1/d} \Big]>0. 
\end{equation}
On the one hand we have that
\[
\Big(\frac{2}{2+1/(d-1)}\Big)^d  < \Big(\frac{1}{1+1/(2d)}\Big)^d \le \Big(\frac67 \Big)^3.
\]
On the other hand we have that
\begin{align*}
	\frac{1}{2+1/(d-1)} - \Big(\frac{2}{2+1/(d-1)}\Big) ^d & = 
	\frac{1}{2+1/(d-1)} \Big[1-2 \Big[ \Big(1+\frac{1}{2(d-1)} \Big)^{2(d-1)} \Big]^{-1/2} \Big]\\ 
	& \ge
	\frac{1}{2+1/(d-1)} \Big[1-2\Big(1+\frac{1}{4}\Big)^{-2}\Big] \ge \frac{-7}{50}, 
\end{align*}
and 
\[
0<1+\Big(3+ \frac{1}{d-1}\Big)^{1/d} < 1+\Big(3+ \frac{1}{2}\Big)^{1/3} = 1+ (7/2)^{1/3}.
\]
Hence, to prove \eqref{desigualtat-perawhD} it is enough to check that
$$
1-\left(\frac{6}{7}\right)^3-\frac{7}{50}\left(1+\left(\frac{7}{2}\right)^{1/3}\right) \approx  0.02 >0.
$$

Moreover, we also claim that $\wh \gamma _{m^-}$ is below $T^{-1}_d(\wh \D) $. This easily follows from the fact that,  by the description of the 
preimage $T^{-1}_d(\wh \D) $, the left boundary of $T^{-1}_d(\wh \D) $ is the graph of a convex function $x=g(y)$ while  $\wh \gamma _{m^-}$
is the graph of a concave function $y=h_{m^-}(x)$ and both graphs are tangent at $p_1$.
 
It follows from lemmas above that we have a deep control on the {\it left} and {\it right} boundaries of $T_d(\wh \D)$, and their relative position with respect to the set $T_d^{-1}(\wh \D)$. See the right picture of Figure  \ref{fig:D_whD}. Now we close the argument by controlling the image of $\partial\wh \D \cap \{y=x\}$.

\begin{lemma} \label{lem:tapadedalt}
The upper piece of the boundary of $T_d(\wh \D) $ is the 
image by $T_d$ of the piece of the boundary 
$
\big\{ (t,t) \mid \frac{1}{1-\wh m^\star}  \le t \le \frac{1}{1-m^-}   \big\}
$
of $\wh \D$. It can be represented as the graph of an increasing function and is contained in $\{ 
x<0, \, y<0\}$. 
\end{lemma}	
\begin{proof}
We introduce
	\[
	T_d(t,t) = (t-(2t)^d, t-2(2t)^d )=:  (\xi(t), \eta(t))  ,\qquad  (1-\wh m^\star)^{-1} \le t \le (1-m^-)^{-1}.
	\] 
	Taking first and second derivatives we have 
$$
\xi'(t)  = 1-2d (2t)^{d-1}, \quad   \eta'(t)   =  1-4d (2t)^{d-1}, \quad \xi''(t)  = -4d(d-1) (2t)^{d-2} \ \  \mbox{\rm and} \  \ \eta''(t)  =  2 \xi''(t).
$$
First we check that, in the corresponding domain, $	\xi'(t)<0$ and $	\eta'(t)<0$. This follows from 
$	\xi''(t)<0, \ \eta''(t)<0$ and  
\begin{align*}
	\xi'\Big(\frac{1}{1-\wh m^\star}\Big)  
	& = 1-2d  \Big(\frac{2}{1-\wh m^\star}\Big)^{d-1} 
	= 1-2d\frac{1}{\Big(1+ \frac{1}{2(d-1)}\Big)^{d-1}} <1-\frac{2d}{\sqrt{e}} <0,
\end{align*}	
and 
\[		\eta'\Big(\frac{1}{1-\wh m^\star}\Big)  
= 	\xi'\Big(\frac{1}{1-\wh m^\star}\Big)  - 2d  \Big(\frac{2}{1-\wh m^\star}\Big)^{d-1}  <0 .
\]
The condition $\xi'(t)<0$ implies that $\xi(t)$ is invertible. Let $t=t(\xi)$ be its inverse function
and
$\eta = f(\xi) := \eta \circ t(\xi)$. The curve $T_d(t,t)$ is the graph of $f$ and 
\[
f^{\prime}=   \frac{\eta^{\prime}}{\xi^{\prime}} \circ t(x)>0.
\]
Moreover, since 
\begin{align*}
		\xi\Big(\frac{1}{1-\wh m^\star}\Big) 
		& = \frac{1}{1-\wh m^\star}\left(1-2\left(\frac{2}{1-\wh m^\star}\right)^{d-1}\right) = \frac{1}{1-\wh m^\star}\left(1-2\frac{1}{(1+ \frac{1}{2(d-1)})^{d-1}}\right) \\
		& < \frac{1}{1-\wh m^\star}\Big(1-\frac{2}{\sqrt{e}} \Big) <0,
	\end{align*}
we have $\xi(t)<0 $ and $\eta (t) = \xi(t) -(2t)^d <0 $.
\end{proof} 

Up to this point we have completed the study of the geometry and relative positions of $T_d(\wh\D)$ and $T_d^{-1}(\wh\D)$ (see the right picture of Figure \ref{fig:D_whD}). Next two lemmas show that there is a piece of $W_{p_1}^s$ attached to $p_1$, being tangent to $y=  m^{-} x-1$ at $p_1$, included in $T_d^{-1}(\wh \D)$ and connecting $p_1$ with a point in $\partial T_d^{-1}(\wh \D) \cap \{y=-x\}, \ x<0$. 

Let $(x_0,y_0) \in T_d^{-1}(\wh \D)$. Then we write $(x_{2k},y_{2k}):=T_d^{2k}(x_0,y_0)$. The first lemma characterize the dynamics of points in $T_d^{-1}(\wh \D)$ whose all iterates under $T_d^{2}$ remain in $T_d^{-1}(\wh \D)$.

\begin{lemma} \label{lem:convergencia-a-p1}
If $(x_0,y_0) \in T_d^{-1}(\wh \D)$  and $(x_{2k},y_{2k}) \in T_d^{-1}(\wh \D)$  
for all $k\ge 0$ then we have that $(x_{2k},y_{2k}) \to p_1=(0,-1)$ as $k \to \infty$.
\end{lemma}
\begin{proof}
	First we note that $y_2<0$. Indeed, 
	$(x_{2},y_{2}) \in T_d^{}(\wh \D)$ and  by Lemma \ref{lem:tapadedalt}
	$$
  \sup \Big\{\eta(t) \mid\ \frac{1}{1-\wh m^\star} \le t \le \frac{1}{1- m^-}\Big\} = \eta\Big(\frac{1}{1-\wh m^\star}\Big) <0.
	$$ 
Moreover, the right boundary of $T_d^{}(\wh \D)$ is given by 
	$\wh  {\gamma}_{\wh m^\star}(t)= (\wh  {x}_m(t),\wh  {y}_m(t))$ and, by Lemma 
	\ref{fitescomponentsgamma}(a), $\wh  {y}_{\wh m^\star} (t) < \wh  {y}_{\wh m^\star} (\frac{1}{1-\wh m^\star}) 
	= \eta(\frac{1}{1-\wh m^\star})<0$.

	Now, let $(x_0,y_0)$ as in the statement  with $y_0<0$. Using that 
	$T_d^{-1}(\wh \D) \cap \{y\le 0 \}$ is below the line 
	$y =-(2-\wh m^\star)^{1/d} x -1$ we have that 
	\begin{equation}\label{fita-x0}
		x_0< \frac{y_0+1}{-(2-\wh m^\star)^{1/d}} . 
	\end{equation}
	
	First, we compute
	\[
	(x_1,y_1) = T_d^{}(x_0,y_0)=(y_0-(x_0+y_0)^d ,y_0-2(x_0+y_0)^d ).
	\]
	We observe that $(x_1,y_1)$ belongs to the line $y=2x-y_0$.	
	
	By the definition of $\wh \D$, we have that $x_1$ is less than the first coordinate 
	of the intersection 
	$\{ y=2x-y_0\} \cap \{ y=   m^{-} x +1\}$, i.e.,  
	$x_1 < \frac{1+y_0}{2-  m^-}$.
	Moreover, using \eqref{fita-x0},  
	\[
	0\le x_1 < \frac{-(2-\wh m^\star)^{1/d}}{2- m^-} x_0.
	\]
	Next we bound 
	\[
	\left| \frac{-(2-\wh m^\star)^{1/d}}{2- m^-}\right| 
	< \frac{(3+\frac{1}{d-1})^{1/d}}{3} \le \frac13 \Big(\frac72\Big)^{1/3}.
	\]
	
Now we deal with the next iterate $(x_2,y_2) =(y_1- (x_1+y_1)^d ,y_1-2(x_1+y_1)^d )$.
Since $(x_1,y_1) \in \wh \D $, 
	$0< x_1+y_1 \le 1$ and $y_1 \ge \wh m^\star x_1 +1$ we conclude that   
	\[
	0\ge x_2 = y_1-(x_1+y_1)^d \ge y_1 -1 \ge \wh m^\star x_1.
	\]
Consequently, 
	\[
	|x_2| \le |\wh m^\star| x_1 \le \frac32 \frac13 \left(\frac72\right)^{1/3}|x_0| \le 
	\frac45 |x_0|.
	\]
Recursively, we  obtain that  $|x_{2k}| \le (\frac45 )^k  |x_0|$  and this implies $x_{2k} \to 0$, Since, by hypothesis, $(x_{2k},y_{2k}) \in T_d^{-1}(\wh \D)$ for all $k \geq 0$  we conclude that $ y_{2k} \to -1$.
\end{proof}


\begin{lemma} \label{lemma:stable_manifold_p1}
The set $T^{-1}_d(\wh \D)$ contains a piece of $W^s_{p_1}$ joining the point $p_1$ with a point in $T^{-1}_d(\wh \D) \cap \{y=-x\}$.
\end{lemma}

\begin{proof} We will use the same argument we have used in \cite{FGJ24}. 
Take $I_0$ any segment joining the right and left boundaries of    
$T^{-1}_d(\wh \D)$. By the previous lemmas, $T^2_d(I_0)$ is a curve contained in  $T^{}_d(\wh \D)$ joining its right and left boundaries which are outside $T^{-1}_d(\wh \D)$, thus it has to cross the right and left boundaries of $T^{-1}_d(\wh \D)$. 

We define $I_1 = T^{-2}_d (T^{2}_d (I_0) \cap T^{-1}_d(\wh \D)) \subset I_0$ and, in general, 
\[
I_n = T^{-2n}_d (T^{2n}_d (I_{n-1}) \cap T^{-1}_d(\wh \D)) \subset I_{n-1}, \qquad n\ge 1.
\] 
Then, $\{I_n\}_{n\ge 1}$ is a sequence of nested compact sets and 
$I_\infty := \bigcap_{n\ge 1}I_n \ne \emptyset$. This set has the property that all points in $I_\infty$  are such that all their iterates stay in 
$T^{-1}_d(\wh \D)$ and, by Lemma \ref{lem:convergencia-a-p1}, converge to $p_1$. Therefore, $I_\infty = W^s_{p_1} \cap T^{-1}_d(\wh \D)\cap I_0$. 
\end{proof}

\begin{proof}[Proof of Proposition \ref{prop:hetero}]
We will see that the above description of the relative positions of $T_d(\mathcal  D)$ and $T_d^{-1}(\wh {\mathcal D})$ 
(neighbourhoods of pieces of $W^u_{p_0}$ and $W^s_{p_1}$, respectively) implies a heteroclinic intersection between the stable manifold of $p_1$ and the unstable manifold of $p_0$. Unless it is necessary, we drop the dependence on the parameter $m$.

On the one hand, in \cite[Lemma 5.4]{FGJ24} it is proven that there is a connected piece of $W^u_{p_0}$ contained in $T_d(\D)$ joining $p_0$ with some point in $\partial T_d(\D) \cap \{y=x\}$.  On the other hand, the above lemmas show that there is a piece of $W^s_{p_1}$ contained in  $T_d^{-1}(\wh \D)$ which joints $p_1=(0,-1)$ with a point in $T_d^{-1}(\wh \D)\cap \{y=-x\}$. 

We claim that the line $L$ given by  $\{y= m^-x-1\}$, tangent to the left boundary of 
$T_d^{-1}(\wh \D)$ at $(0,-1)$, intersects in two points the right boundary of 
$T_d(D)$ which is given by the curve 
$\gamma_{m}(t)$ in
\eqref{gamma-m} with $m=m^\star=7/2$. If we write $\gamma_{m^\star}(t)=\left(X(t),Y(t)\right)$ we have 
\[
{\b X(t)} = m^\star t -1 - ( (m^\star+1)t -1)^d  , \quad 
{\b Y(t)} = m^\star t -1 - 2( (m^\star+1)t -1)^d, \quad t\in \Big[0, \frac{1}{m^\star+1}\Big]. 
\] 
See Figure \ref{fig:unbounded_global_2}. To check the claim, recall that $ \frac{-6}{\sqrt{13}+1} \le  m^-< -1$. 
We consider the auxiliary function 
\begin{align*}
	\phi(t) & = Y(t) - m^-  X(t) +1 \\
	& =(2-  m^-) [(1- (m^\star +1)t )^d + m^\star t -1] 
	-m^\star t +2 ,
	\qquad t\in \Big[0, \frac{1}{m^\star+1}\Big],
\end{align*}
which measures whether  $\gamma_{m^\star}(t)$ is below, above or on the line $L$.
We have 
\begin{equation*}
\begin{split}
& \phi(0)= 2>0, \qquad \phi( \frac{1 }{m^\star +1})= \frac{m^- +m^\star }{m^\star +1}>, \qquad \text{ and} \\
&\phi''(t) = (2- m^-  ) d(d-1)(m^\star +1)^2 (1- (m^\star +1)t)^{d-2}>0, \qquad  t\in \Big[0, \frac{1}{m^\star +1}\Big]. 
\end{split}
\end{equation*}

Accordingly, in order to see that $\phi$ has two zeros in its domain it is enough to show that there is a point $t_1$ in
$(0, \frac{1}{1+m^\star}) $ such that $\phi(t_1)<0$. We take $t_1=1/8$ and, using that
$m^- > - 4/3$,  we have  
\[\phi(1/8)=(2-  m^-) 
\left( \left(\frac{7}{16} \right) ^{d} -\frac{9}{16}\right)  +  
\frac{25}{16}
< \frac{10}{3} \left(\frac{7^3}{16^3}-\frac{9}{16}\right)
+\frac{25}{16}<0.
\]  
Therefore, $W^u_{p_0}$ has to cross  $L$.

Next we claim that 
$T_d^{-1}(\wh \D) \cap \{y=0\}$ is a  segment $[a_-, a_+] \times \{0\}$ with $a_+ < -3/5$. To see this claim we look for the intersection of the right and left boundaries of $T_d^{-1}(\wh D)$, given by $\Gamma_{\wh m^\star }(t) =( \wh  {\alpha}_{\wh m^\star}(t), \wh  {\beta}_{\wh m^\star }(t) )$
and $\Gamma_{m^-}(t) =( \wh  {\alpha}_{m^-}(t), \wh  {\beta}_{m^-}(t) ) $, respectively, with  $\{y=0\}$. We recall that $\wh m^\star =-1-1/(d-1)$ and
\[
\wh  {\alpha}_{\wh m^\star }(t) = ({\wh m^\star } -2) t +1 + ( (1- {\wh m^\star })t -1)^{1/d}  , \qquad 
\wh  {\beta}_{\wh m^\star }(t) =(2-{\wh m^\star }) t -1, \qquad t\in \Big[0, \frac{1}{1-\wh m^\star }\Big]. 
\] 
The value $t=t_2$ such that $\wh  {\beta}_{\wh m^\star }(t) =0$ is 
$t_2= \frac{1}{2-\wh m^\star } \in  \big[0, \frac{1}{1-\wh m^\star }\big]$,
and

$$
a_+=
\wh  {\alpha}_{\wh m^\star }(t_2) =- \left(\frac{1}{2-\wh m^\star }\right)^{1/d} =
- \left(\frac{1}{3 + 1/(d-1) }\right)^{1/d} 
\le 
-\left(\frac{2}{7}\right)^{1/3} <  -\frac35.
$$
In the same way, denoting $t_3$ the value such that 
$\wh  {\beta}_{m^-}(t_3)=0$, we obtain 
$a_-= - \left(2-\wh m^\star \right)^{-1/d} < a_+ $

Putting together the information of the two previous claims we get that when $y=0$, $W^u_{p_0}$  
is to the right of 
$W^s_{p_1}$ and that there exists some $y = y^0<0$ for which $\gamma_{m^{\star}}$ is to the left of $L$ and therefore  
$W^u_{p_0}$ has to be at the left of 
$W^s_{p_1}$. This finish the proof of the proposition. See Figure \ref{fig:unbounded_global_2}.
\end{proof}

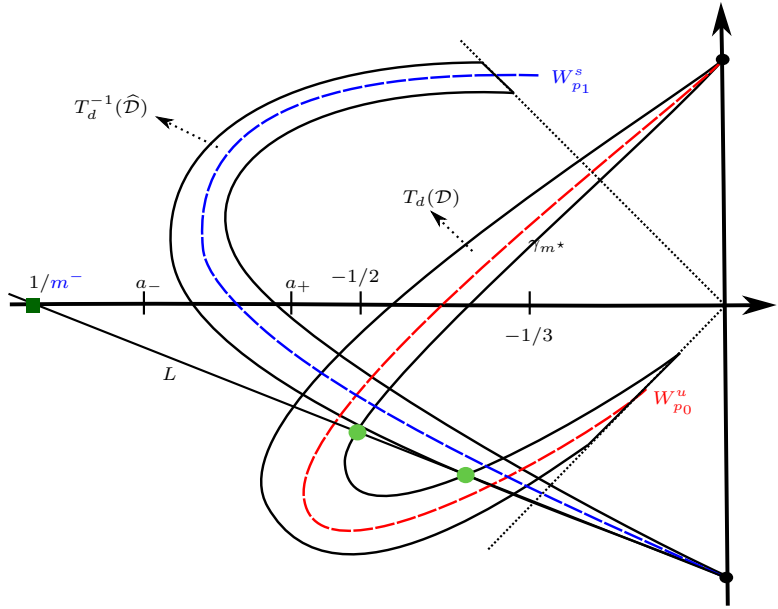
\begin{figure}[ht]
	\centering
	\includegraphics[width=0.65\textwidth]{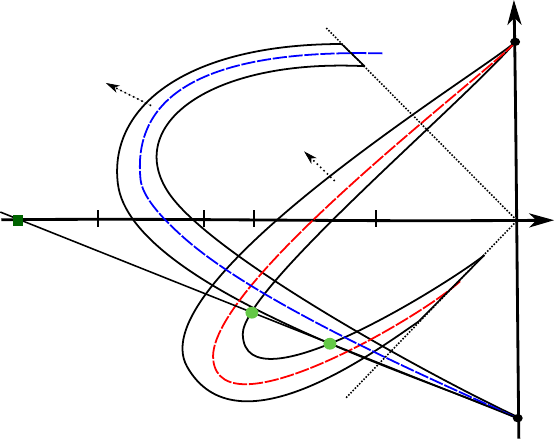}
	\put(-88,197){ \b{\tiny $W^s_{p_1}$}}
	\put(-50,76){ \r{\tiny $W^u_{p_0}$}}
	\put(-93,135){\tiny  $\gamma_{m^\star}$}
	\put(-230,85){\tiny  $L$}
	\put(-240,120){\tiny  $a_{-}$}
	\put(-280,120){\tiny  $1/ {\b m^-}$}
	\put(-184,120){\tiny  $a_{+}$}
	\put(-167,122){\tiny  $-1/2$}
	\put(-102,100){\tiny  $-1/3$}
	\put(-263,186){\tiny  $T_d^{-1}(\wh\D)$}
	\put(-140,152){\tiny  $T_d(\D)$}
	\caption{\small{Sketch of the arguments providing the (topological, not necessarily transversal as it is shown in the picture) intersection of the stable and unstable manifold of the hyperbolic two-cycle $\{p_0,p_1\}$ (Theorem A). The green dots indicates the two intersections between $\gamma_{m^\star}(t)$ and the line $L:=\{y= m^{-}x-1\}$.}}       
	\label{fig:unbounded_global_2}
\end{figure}

\begin{proof}[Proof of Theorem A] Since $d$ is odd, the map $T_d$ is symmetric with respect to $(x,y)\mapsto (-x,-y)$. Proposition \ref{prop:hetero} provides a
(maybe non-transversal) heteroclinic point $q$ in $T_d^{-1}(\wh \D) \cap T_d(\D)$. In any case at this point the manifolds cross each other. Therefore $\overline{q}=-q$ is also a heteroclinic point. By symmetry, at the point $\overline{q}$ the unstable manifold of $p_1$ intersects the stable manifold of $p_0$. We know that the unstable manifold is analytic. The stable manifold is analytic in a neighbourhood of $q$ since the globalization of the local manifold has not meet $\{y=x\}$ yet. Since the  manifolds do not coincide, they have a finite order contact.

Since we do not know if the intersection is transversal, we cannot apply the $\lambda$-Lemma of Palis in \cite{PALIS69}. However, we can apply the singular $\lambda$-Lemma in \cite{Ray}.
In the two dimensional case it asserts that the iteration of a disc in the unstable manifold 
accumulates in a $C^1$ manner to the unstable manifold of $p_0$, except for an arbitrarily small neighbourhood of $p_0$. 

Now consider a piece of the connected component of the unstable manifold of $p_0$ in $T_d^{-1}(\wh \D) \cap T(\D)$ joining two points of the upper and lower boundaries of $T^{-1}(\wh \D)$, respectively. 
Then, by the singular $\lambda$-Lemma, the unstable manifold of $p_1$ will have discs arbitrary
$C^1$-close to the unstable manifold of $p_0$ and therefore the discs will be in $T_d^{-1}(\wh \D) \cap T_d (\D)$.                                       

Finally, using the same argument as in the end of the proof of the first part of Theorem A these discs should have an intersection with the stable manifold of $p_1$ thus providing the desired homoclinic point. 
\end{proof}


\section{Proof of Theorem B} \label{section:proof_B}

In the previous section we have proven the existence of homoclinic points associated to the stable and unstable of the cycle $\{p_0,p_1\}$. Using this fact, in this section we demonstrate that stable and the unstable manifolds of $p_1$ intersect in a transverse homoclinic point.  Our approach is inspired in the work of \cite{CR}. However, there is an important difference. In \cite{CR} the authors deal with analytic area preserving maps and can use tools as the Birkhoff normal form, while our map is not area preserving and it is not an analytic diffeomorphism. Our presentation uses the special structure of the map and the fact that we can linearize the map $T^2_d$ around $p_1$ which a $C^\infty$ conjugation.

For the point $p_1=(0,-1)$ we will denote by 
$$
W_{\loc}^s:=W_{\loc,p_1}^s, \qquad W_{\loc}^u:=W_{\loc,p_1}^u,  \qquad W^s:=W_{p_1}^s  \qquad \text{ and} \qquad  W^u:=W_{p_1}^u
$$
the local stable, local unstable, global stable and global unstable manifolds associated to $p_1$ for the map $T_d^2$, respectively. The size of the local manifolds will be as small as we need.

We split the proof of Theorem B into several  lemmas. Given $z\in \R^2$ we let $B_\varepsilon(z)$ be  the open ball  centered at $z$ and radius $\varepsilon>0$.

\begin{lemma} \label{lemma:q1_q2_local}
Let $\varepsilon >0$ be small enough. Then, there exist two points $q_s$ and $q_u$ in $ B_\varepsilon(p_1)$ such that
$$
q_s \in W_{\loc}^s \cap W^u \qquad \text{and} \qquad  q_u \in 
W_{\loc}^u \cap W^s. 
$$
 Moreover, there exist analytic  local parametrizations of  $W^s$ around $q_u$ 
 and of $W^u$ around $q_s$ 
 given by $\{\phi^s(u)\mid \, u\in (-\delta, \delta)\}$ with $\phi^s(0)=q_u$ 
 and  $\{\phi^u(u)\mid \,   u\in (-\delta, \delta)\}$ with $\phi^u(0)=q_s$
 for some $\delta >0$ small. 
 
 Since the manifolds do not coincide,  the above intersections (at the points $q_s$ and $q_u$) have finite order contact. 
\end{lemma}

Fix $\varepsilon_1>0$ small enough such that 
$B_{\varepsilon_1}(q_s)\subset B_{\varepsilon}(p_1)$ and $B_{\varepsilon_1}(q_u)\subset B_{\varepsilon}(p_1)$. We denote by $\wh{W}^u$ the piece of $W^u \subset B_{\varepsilon_1}(q_s)$ and by  $\wt{W}^s$ the piece of $W^s \subset B_{\varepsilon_1}(q_u)$.

\begin{proof}
Fix $\varepsilon>0$ small enough and consider local manifolds $W_{\loc}^s,  \ W_{\loc}^u $ contained in $ B_\varepsilon(p_1)$. Let $q\in W^s \cap W^u$ be the point determined by the topological transversal intersection of the stable and the unstable manifolds of  $p_1$ for the map $T_d^2$ given by Theorem A. By iterating forward this point by $T_d^2$ and $(T_d^2)^{-1}$ we obtain the existence of $q_s$ and $q_u$ in $B_\varepsilon(p_1)$, respectively. Moreover, since $W_{\loc}^s$ and  $W^u$ are analytic we have that $W_{\loc}^s \cap W^u$ intersect with finite order contact (otherwise they would coincide). Then, there exists $\phi^u$ as claimed. By construction, there exists $n_0>0$ such that
$$
T_d^{-n_0}(q_s)=q_u.
$$
According to the previous arguments if
\begin{equation}\label{eq:y=x}
T_d^{-j}(q_s) \cap \{y=x\} = \emptyset, \qquad j=1,\ldots, n_0-1,
\end{equation}
then $W_{\loc}^u \cap W^s$ intersect at $q_u$, $W^s $ is analytic in a neighbourhood of $q_u$ and the intersection has a finite order contact and the lemma follows. Now, we consider the case that there exists a finite sequence of natural numbers $0< j_1 < j_2 < \cdots < j_\ell <n_0$, $1\le \ell  < n_0$,  such that
\begin{equation}\label{eq:y=x_bis}
T_d^{-j_k}(q_s) \cap \{y=x\}=:{\bf  r}_k \in \mathbb R^2, \qquad k=1,\ldots, \ell .
\end{equation}
Note that ${\bf  r}_k=(r_k,r_k), \ r_k \in \mathbb R$, and hence, $T_d^{-1}(r_k,r_k)=(-r_k,r_k)$. First, we deal with ${\bf r}_1$, the first time the globalization of $W^s_\loc$  meets $\{y=x\}$ so that $W^s$ is analytic from $p_1$ to this point. Thus, near ${\bf  r}_1$ the stable manifold $W^s$,
is analytic and  can be parametrized  as
$$
\phi(t)=\left(r_1+t^{\alpha_1}\left(a_1+f_1(t)\right),r_1+t^{\beta_1}\left(b_1+g_1(t)\right)\right), \qquad |t| < \delta_1,  
$$
where $\alpha_1,\beta_1 \in \mathbb N$, $a_1,b_1 \in \mathbb R\setminus \{0\}$, $f_1(t),g_1(t)$ are analytic, satisfy  $f_1(0)=  0$ and $g_1(0)=  0$ and  $\delta_1>0$ is small enough. Since $W^s \not\subset \{y=x\}$ we have 
$$
t^{\alpha_1}\left(a_1+f_1(t)\right) - t^{\beta_1}\left(b_1+g_1(t)\right) \not\equiv 0.
$$
Using the expression of $T_d^{-1}$ (see Eq. \ref{eq:T_inverse_odd})  we have
$$
T_d^{-1}\left(\phi(t)\right)=\left(\begin{array}{l}
-r_1-2t^{\alpha_1}\left(a_1+f_1(t)\right)+t^{\beta_1}\left(b_1+g_1(t)\right)+t^{\gamma_1/d}\left(1+O(t)\right)^{1/d}\\ 
\ \ r_1+2t^{\alpha_1}\left(a_1+f_1(t)\right)-t^{\beta_1}\left(b_1+g_1(t)\right)
\end{array}\right), 
$$
where $\gamma_1\ge \min\{\alpha_1,\beta_1\}$.
Since $d$ is odd we can  reparametrize the curve $\phi(t) $ using the new parameter $u=t^{1/d}$ to   obtain $\wh {\phi}(u) =\phi(u^{d}) $ analytic and
$$
T_d^{-1}\left(\wh {\phi}(u)\right) 
= T_d^{-1}\left(\phi(u^d)\right)=\left(\begin{array}{l}
-r_1+O(u^{\wh {\alpha}_1})\\ 
\ \ r_1+O(u^{\wh {\beta}_1})
\end{array}\right), 
$$
with $\wh {\alpha}_1$, $\wh {\beta}_1 \in \N$.

We conclude thus  that $W^s$ admits an analytic parametrization in a sufficiently small neighbourhood of  $T_d^{-1}(r_1,r_1)=(-r_1,r_1)$. Repeating  the same procedure a finite number of times it is clear that $W^s$  intersects  $W_{\loc}^u$  with finite order contact at the point $q_u$.
\end{proof}

The translation  
\[
\T: (\wh  x,\wh  y)\mapsto (x,y)=(\wh  x,\wh  y-1)
\] 
moves $p_1$ to the origin. For simplicity we write the new coordinates again as $(x,y)$. Observe that (in the new coordinates) $T_d^2(0,0)=(0,0)$ and that 
\begin{equation}\label{eq:linear_matrix_p1}
DT_d^2(0,0)=\left(\begin{array}{cc}
3d^2-2d &\ \  3d^2-4d+1 \\
6d^2-2d & \ \ 6d^2-6d+1
\end{array}\right).
\end{equation}
The eigenvalues and eigenvectors are given in 
\eqref{eq:lambda} and \eqref{eq:m}, respectively.

We will denote 
\begin{equation} \label{deflambdamu}
	\lambda:=\lambda_d^{+}> 1,\qquad \mu:=\lambda_d^{-} < 1, \qquad  m_{\lambda}:= m^{+}_d \qquad \text{ and} \qquad m_{\mu}:= m^{-}_d
\end{equation}
 (we drop the dependence on $d$ unless it is strictly necessary). We recall from Section \ref{sect:prelimpaper1} that 
\begin{equation} \label{fiteslambdamu}
\lambda> 57, \quad 1/9<\mu<0.1556, \quad 2<m_\lambda <2.3028 \quad \text{and} \quad -1.3027<m_\mu <-1.
\end{equation}

We parametrize the local stable and unstable manifolds associated to the origin by the $x$-variable so that the expressions can be written as $y=\Psi^s(x)$ and  $y=\Psi^u(x)$, respectively. We obviously have 
\begin{equation}\label{eq:Phi_s_u}
  \frac{d\Psi^s}{dx}(x)|_{x=0} =m_\mu  \qquad \text{ and} \qquad   \frac{d\Psi^u}{dx}(x)|_{x=0} = m_\lambda .
\end{equation}

Next step is to introduce local analytic coordinates $( \wh {\xi},\wh {\eta}) $
around $(0,0)$ so that the expression of the local stable and unstable manifolds would be $\wh {\eta}=0$ and $\wh {\xi}=0$, respectively. 

\begin{lemma} \label{lemma:tothom_recte}
We consider the local change of variables   
$$(
x,y)\mapsto (\wh {\xi},\wh {\eta})=\Theta(x,y):=( y-\Psi^u(x), y-\Psi^s(x)). 
$$
Then, for $\varepsilon>0$ small enough the local expression of $T_d^2$ in $B_\varepsilon(0,0)$ is given by
$$
\mathcal F(\wh {\xi},\wh {\eta})=\mathcal L(\wh {\xi},\wh {\eta}) + \mathcal N(\wh {\xi},\wh {\eta}), 
$$
where 
$$
\mathcal L(\wh {\xi},\wh {\eta})=(\lambda\wh {\xi}, \mu\wh {\eta}),\quad \mathcal N(0,0)=(0,0) \quad \text{ and} \quad D\mathcal N(0,0) = 0.
$$
Moreover, the local change of coordinates, $ (\wh {\xi},\wh {\eta})=\Theta(x,y)$, is analytic.
\end{lemma}

\begin{proof}
We claim that the new variables define a local change of coordinates around the origin. Indeed, since $\Psi^s(x)$ and  $\Psi^u(x)$ are analytic,  by the Inverse Function Theorem we only need to check that 
$$
D\Theta(0,0)
=
\left(\begin{array}{cc}
- \frac{\partial \Psi^u}{\partial x}(x)|_{x=0} & 1 \\ 
 & \\
- \frac{\partial\Psi^s}{\partial x} (x)|_{x=0} & 1 
\end{array}\right)
=
 \left(\begin{array}{cc}
		- m_\mu & 1 \\ 
		& \\
		- m_\lambda & 1 
	\end{array}\right)
$$
is non-singular and this is a direct consequence of \eqref{fiteslambdamu}. 
Clearly $\Theta$ is a local analytic diffeomorphism.  See Figure \ref{fig:reparametrization}.
\end{proof}

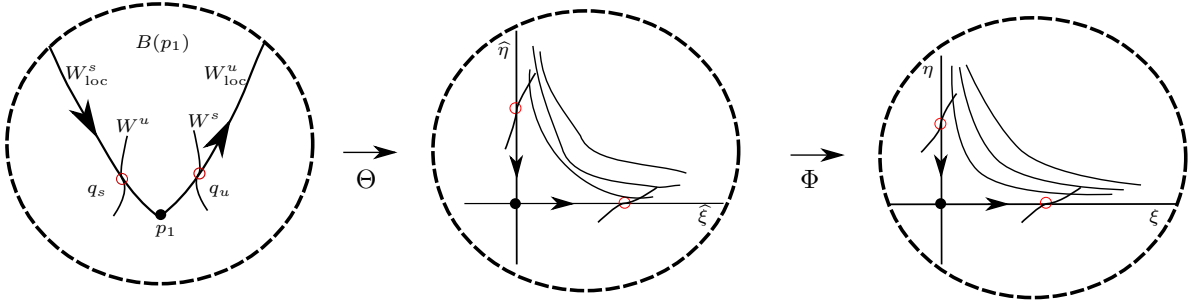
\begin{figure}[ht]
    \centering
     \includegraphics[width=\textwidth]{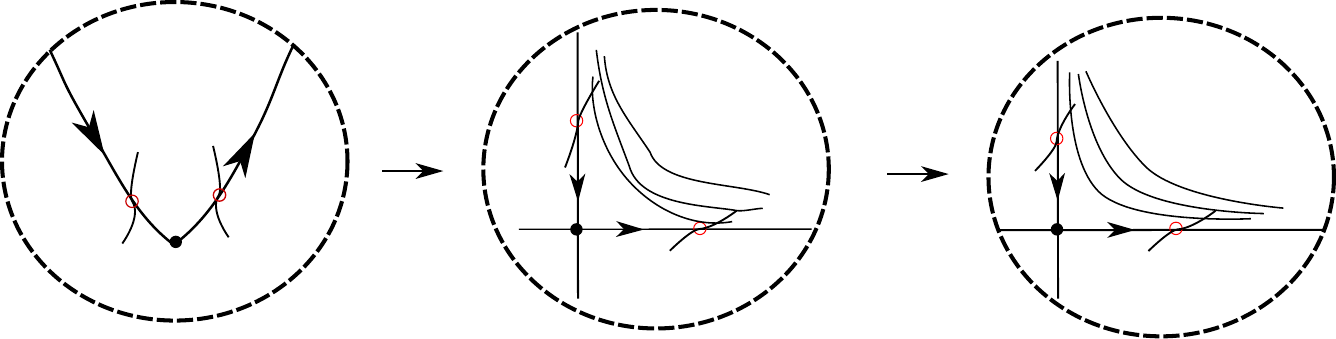}
      \put(-372,83){ \tiny $W_{\loc}^u$}
       \put(-424,83){ \tiny $W_{\loc}^s$}
       \put(-405,63){ \tiny $W^u$}
       \put(-378,65){ \tiny $W^s$}
         \put(-412,40){\tiny $q_s$}
       \put(-370,40){ \tiny $q_u$}
       \put(-390,25){ \tiny $p_1$}  
        \put(-394,95){\tiny $B(p_1)$}
         \put(-315,42){ \small $\Theta$}
            \put(-145,42){\small $\Phi$}
     \put(-187,29){ \tiny $\wh {\xi}$}
       \put(-262,92){ \tiny $\wh {\eta}$}
         \put(-18,29){ \tiny $\xi$}
           \put(-103,85){ \tiny $\eta$} 
                \caption{\small{The changes of coordinates corresponding to Lemma \ref{lemma:tothom_recte} and Lemma \ref{lemma:sternberg}. In fact $\Theta$ in this figure includes a primer change of coordinates to move $p_1$ to the origin.}}       
    \label{fig:reparametrization}
    \end{figure}

We will  see next that $\mathcal F$ is $\mathcal C^{\infty}$-conjugate to its linear part $\mathcal L$. For this we will apply Sternberg's Theorem \cite[Theorem 1]{Ste}. The following lemma  checks a  key  hypothesis of that theorem.

\begin{lemma} \label{lemma:resonancies}
The eigenvalues $\lambda =  \lambda_d^{+}>1 $ and $\mu=\lambda_d^{-}<1$ of the linear part $\mathcal L$ of $\mathcal F$ at $(0,0)$ given in \eqref{deflambdamu} and \eqref{eq:lambda} are non-resonant.
\end{lemma}

\begin{proof} 
We first note that $\lambda\mu=d^2$ (the determinant of $DT^2_d(0,0)$). It is easy to check by induction that    
$$
\lambda^k=a_k+b_k\sqrt{\Delta} \qquad \text{ and} \qquad  \mu^k=a_k-b_k\sqrt{\Delta}, \qquad k\ge 1,
$$
and 
$$
\lambda^{-k}=\frac{1}{g^k} \left(a'_k-b'_k\sqrt{\Delta}\right) \qquad \text{ and} \qquad   \mu^{-k} =\frac{1}{g^k} \left(a'_k+b'_k\sqrt{\Delta}\right), \qquad k\ge 1,
$$
where $a_k,b_k, a'_k, b'_k \in \mathbb N$, $g=12 d^2 (6 d^2 - 4 d + 1)\in \N $  and $\Delta=9d^2-10d+1$ is not a perfect square. 
This means that $\lambda^k, \mu^k \in \R\setminus \Q $ for all $k\ne 0$.

There are two possible types of resonances: 
\begin{equation} \label{eq:lambda_resonances}
	\lambda=\lambda^{n}\mu^{m} \qquad \text{ with} \qquad n,m\geq 0 \qquad 
	\text {and} \quad n+m\ge 2, 
\end{equation}
and 
\begin{equation} \label{eq:mu_resonances}
	\mu=\lambda^{n}\mu^{m} \qquad \text{ with} \qquad  n,m\geq 0 \qquad 
\text {and} \quad n+m\ge 2. 
\end{equation}

We deal with \eqref{eq:lambda_resonances}. We rewrite   it  as
\begin{equation} \label{eq:lambda_resonances-mod}
	1= \lambda^{n-m-1}(\lambda\mu)^{m} =\lambda^{n-m-1} d^{2m}  .
\end{equation}
We distinguish two cases: 
(a) $n\ne m+1$  and (b)   $n= m+1$.

In case (a), since $\lambda^{n-m-1} \in \R\setminus \Q $ and $d\in \N$ the previous equality is impossible.
In case (b),  $m$ cannot be 0. Then, $(\lambda\mu)^{m}= d^{2m} \ge 9$ so that \eqref{eq:lambda_resonances-mod}  is also impossible. 

Concerning resonances of the form \eqref{eq:mu_resonances} the argument is completely analogous.

\end{proof}

Theorem 1 in \cite{Ste} provides a $\mathcal C^{\infty}$ local change of coordinates conjugating $\mathcal F$ to its linear part $\mathcal L$. 
From it we will obtain a near the identity conjugation.

\begin{lemma}\label{lemma:sternberg}
There is a conjugacy from $\mathcal F$ to its linear part $\mathcal L$ at the origin of the form
\begin{equation}\label{sternberg}
(\xi,\eta):=\Phi(\wh {\xi},\wh {\eta})=\left(\begin{array}{c}
\wh {\xi} (1+\phi_1(\wh {\xi},\wh {\eta})) \\ 
\wh {\eta} (1+\phi_2(\wh {\xi},\wh {\eta}))
\end{array}\right), 
\end{equation}
where $\mathcal \phi_j(\wh {\xi},\wh {\eta}),\ j=1,2$,  are $\mathcal C^{\infty}$  functions defined in a sufficiently small neighbourhood of the origin with $\mathcal \phi_j(0,0)=(0,0),\ j=1,2$.
\end{lemma}

\begin{proof}
Let $\wh {\Phi}(\wh {\xi},\wh {\eta})$ be the $\mathcal C^{\infty}$ local conjugacy given by Sternberg's Theorem. Consequently,  $\wh {\Phi}$ should send the stable and unstable manifolds of $\mathcal F$ to the corresponding ones of $\mathcal L$, which in this case means that  it  preserves the axes.
Writing $\wh {\Phi} = (\wh {\Phi}_1,  \wh {\Phi}_2)$, this is translated into the conditions $\wh {\Phi}_1(0,\wh {\eta}) =0$ and $\wh {\Phi}_2(\wh {\xi},0)=0$. Then,
\[
\wh {\Phi}_1(\wh  \xi,\wh {\eta}) 
= \wh {\Phi}_1(0,\wh {\eta})+ \int _0^1 \partial _\xi \wh {\Phi}_1(t \wh  \xi,\wh {\eta})\, \xi \, dt 
= \wh  \xi (\alpha + \wh {\phi}_1(\wh  \xi,\wh {\eta}) ) 
\]
with ${\wh \phi}_1(0,0)=0$, 
and analogously 
$
\wh {\Phi}_2(\wh  \xi,\wh {\eta}) 
= \wh  \eta (\beta + \wh {\phi}_2(\wh  \xi,\wh {\eta}) ) 
$ with  ${\phi}_2(0,0)=0$.
Since $\wh {\Phi}$ is a diffeomorphim, $\alpha \beta \ne 0$. 
We write $A = \left(\begin{array}{cc}\alpha & 0 \\0 & \beta\end{array}\right) $. 
We claim that $\Phi:=A^{-1}\wh {\Phi}$ is also a conjugation from $\mathcal F$ to $\mathcal L$. Indeed, since $A$ commutes with $\mathcal L$,
$$
\Phi \mathcal F=A^{-1}\wh {\Phi}\mathcal F = A^{-1} \mathcal L \wh {\Phi}=\mathcal L A^{-1} \wh {\Phi}=\mathcal L \Phi.
$$
Moreover,  $\Phi$ is of the form given in  \eqref{sternberg}.  See Figure \ref{fig:reparametrization}.

\end{proof}

In Lemma \ref{lemma:q1_q2_local} we have proven the existence of the points 
$$
q_s\in W_{\loc}^s\cap \wh{W}^u \qquad \text{ and} \qquad q_u \in W_{\loc}^u\cap \wt{W}^s.
$$ 
Then, we can use the changes of coordinates introduced in the previous lemmas to transport those curves to a neighbourhood of the origin. Denote by $\gamma_1(t)$ and $\gamma_2(t)$ the parametrizations of $(\Phi \circ \Theta \circ \T)(\wh{W}^u)$ and $(\Phi \circ \Theta \circ \T )(\wt{W}^s)$, respectively. We  focus on the pieces of $\gamma_1(t)$ and $\gamma_2(t)$ in the first quadrant. Without loss of generality we can assume that these pieces are parametrized by $t\ge 0$.

\begin{lemma}\label{lemma:curves}
The curves $\gamma_1(t)$ and $\gamma_2(t)$ intersect   the coordinate axes $\{\xi=0\}$ and $\{\eta=0\}$ 
at points $\wh q_s = (0,\eta_0)$ and $\wh q_u=(\xi_0, 0)$	
and have a finite order contact there, respectively. Moreover, for $t$ small enough we have that $\gamma_j(t),\ j=1,2$,  admit the following parametrization
\begin{equation} \label{eq:gammes_t}
\begin{split}
&\gamma_1(t)=(t^{\ell_1}(a_1+g_1(t)),\eta_0+t), \\
&\gamma_2(t)=(\xi_0+t^{\ell_2}(a_2+g_2(t)),t^{\ell_3}(a_3+g_3(t))),
\end{split}
\end{equation}
where  $\ell_j\geq 1$  for $j=1,2,3$, $a_1 a_2 a_3\ne 0$, $\xi_0,\eta_0>0$, and $g_j(t)$ are $C^\infty$ functions with $g_j(0)=0$, for $j=1, 2,3$. 
See Figure \ref{fig:transversal_intersection}.
\end{lemma}

\begin{proof}
The lemma follows from the fact that $q_s$ and $q_u$ are points of finite order contact between the stable and the unstable manifolds of $p_1$ and the change of coordinates we have used is $\mathcal C^{\infty}$. 
\end{proof}

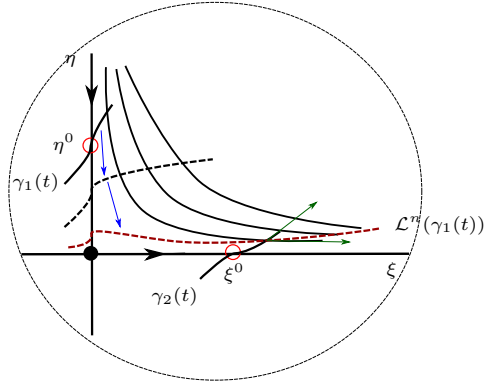
\begin{figure}[ht]
    \centering
     \includegraphics[width=0.35\textwidth]{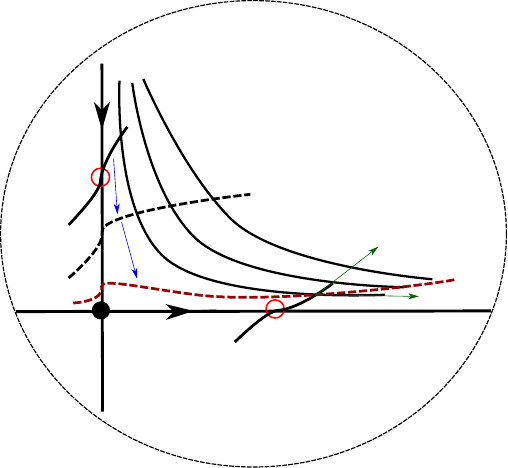}
      \put(-157,73){ \tiny $\gamma_1(t)$}
      \put(-142,87){ \tiny $\eta^0$}
      \put(-134,120){\tiny $\eta$}
      \put(-14,57){ \tiny $\mathcal L^n(\gamma_1(t))$}
      \put(-105,30){ \tiny $\gamma_2(t)$}
      \put(-78,38){ \tiny $\xi^0$}
      \put(-17,40){ \tiny $\xi$}
 \caption{\small{Sketch of the situation described in Lemma \ref{lemma:curves}.}}       
    \label{fig:transversal_intersection}
 \end{figure}

We want to show that, for $k$ large enough, $\gamma_2(t)$ and $\mathcal L^k(\gamma_1(t))$ intersect transversally. This framework is quite close to the one in \cite[Theorem 1.1]{CR}  but in their case the linear map admits  the function $H(x,y)= xy$ as a first integral, which is not our case. Then, we provide a proof in our case to get  the same conclusion.

Given $\lambda>1 $ and $0<\mu<1$ introduced in \eqref{deflambdamu}
and \eqref{eq:lambda} and we
consider the auxiliary interpolation map
$$
\mathcal L^\tau(\xi,\eta)=\left(\begin{array}{cc} \lambda^\tau & 0  \\ 0  & \mu^\tau \end{array}\right)  \left(\begin{array}{c} \xi   \\ \eta  \end{array}\right)   
, \qquad \tau>0.
$$
We also consider the first quadrant  $Q= \{(\xi,\eta) \mid \ \xi\ge0, \eta\ge 0\}$ and $H:Q \to Q$ defined by 
\begin{equation}\label{eq:H}
	H(\xi,\eta)=\xi^{\log \mu^{-1}}\eta^{\log \lambda}.
\end{equation}
It is continuous and real analytic in the interior of $Q$.

\begin{lemma} \label{lemma:level_curves}
The function $H$ is a first integral of $\mathcal L^\tau$, $\tau >0$, in   $Q$.
\end{lemma}

\begin{proof}
To prove the lemma we compute 
\[
H\left(\mathcal L^\tau(\xi,\eta)\right)
=H\left(\lambda^\tau\xi,\mu^\tau\eta\right)
=\lambda^{\tau\log\mu^{-1}}\mu^{\tau\log\lambda} \xi^{\log \mu^{-1}}\eta^{\log \lambda}
=	H(\xi,\eta).
\]
\end{proof}

Next step is to show that there exist  reparametrizations $t=\sigma_j(s)$,  $j=1,2$, 
of the curves $\gamma_j(t),\ j=1,2$, which have a useful property.

\begin{lemma} \label{lemma:reparam_gamma_j}
There exist continuous reparametrizations $ \wt \gamma_j (s)=\gamma_j(\sigma_j(s))$, $s\in [0,s_0)$,  of the curves $\gamma_j(t)$ given by  $t=\sigma_j(s)$, $ j=1,2$, that are $C^\infty$ in $(0,s_0)$  and they satisfy
$$
H\left(\wt{\gamma}_j (s)\right)=s, \qquad  s\in (0,s_0)
$$
for some $s_0>0$ small enough. 
\end{lemma}
\begin{proof}
For $j=1$ we impose the condition  $H\left(\gamma_1(t)\right)=s$ to obtain  $t:=\sigma_1(s)$. Using \eqref{eq:gammes_t} and \eqref{eq:H} we have
\begin{equation}
\begin{split}
H\left(\gamma_1(t)\right)=s \quad &\iff \quad t^{\ell_1\log \mu^{-1}}(a_1+g_1(t))^{\log\mu^{-1}}\left(\eta_0+t\right)^{\log \lambda}=s \\
& \iff  G_1(t):=t(a_1+g_1(t))^{1/\ell_1}\left(\eta_0+t\right)^{\log \lambda / (\ell_1 \log \mu^{-1})}=s^{1/\left(\ell_1\log\mu^{-1}\right)}.
\end{split}
\end{equation}
We have $G_1(0)=0$, $ c_1:=G_1'(0)=a_1^{1/\ell_1}\eta_0^{\log \lambda / (\ell_1 \log \mu^{-1})} \ne 0$.  

Consequently, by the Inverse Function Theorem,  $G_1$ is locally invertible
and we can write 
\begin{equation}\label{eq:G_1}
t=\sigma_1(s):=G_1^{-1} (s^{1/(\ell_1\log\mu^{-1})})
\end{equation}
for $|s| < s_0 $ for some $s_0 >0$ small.
 
For $j=2$, arguing as above, we have
\begin{equation*}
H\left( \gamma_2(t) \right)=s \ \iff \ G_2(t):=t(a_3+g_3(t))^{1/\ell_3}\left(\xi_0+t^{\ell_2}\left(a_2+g_2(t)\right)\right)^{\log \mu^{-1} / (\ell_3 \log \lambda)}=s^{1/\left(\ell_3\log \lambda\right)}.
\end{equation*}
Analogous computations imply that  $c_2:= G_2'(0)=a_3^{1/\ell_3}\xi_0^{\log \mu^{-1} / (\ell_3 \log \lambda)} \ne 0$ and
\begin{equation}\label{eq:G_2}
t=\sigma_2(s):=G_2^{-1} (s^{1/\left(\ell_3\log \lambda\right)}).
\end{equation}  
\end{proof}

The following lemma establishes a relation between $ \mathcal L^\tau\left(\wt{\gamma}_1(s)\right)$ and $\wt{\gamma}_2(s)$. 

\begin{lemma} \label{lemma:moving_Lrho}
Let $s_0>0$ be small enough. Then, there exists  a $C^\infty$ function $\tau(s), \ s\in(0,s_0)$, such that 
$$
\mathcal L^{\tau(s)} \left(\wt{\gamma}_1(s)\right)=\wt{\gamma}_2(s), \qquad  s\in(0,s_0).
$$
Moreover 
\begin{equation} \label{eq:hro_0_lim}
\lim_{s\to 0^+} \tau(s)=\infty.
\end{equation}
In particular, there exist  $k_0$ in $\mathbb N$  and  a sequence  of positive values  $\{s_k\}_{k\ge k_0}$, such that $s_k\to 0 $ and   $\tau(s_k)=k$ for every $k \geq k_0$.
\end{lemma}

\begin{proof}
Let $s_0>0$ be as in Lemma \ref{lemma:reparam_gamma_j}. We use the following notation 
$$
\wt{\gamma}_1(s)=(\wt{\xi}_1(s),\wt{\eta}_1(s)) \qquad \text{and} \qquad \wt{\gamma}_2(s)=(\wt{\xi}_2(s),\wt{\eta}_2(s)). 
$$
We define
$$
\tau(s):=\frac{1}{\log \lambda}\log\frac{\wt{\xi}_2(s)}{\wt{\xi}_1(s)}, \qquad s\in (0,s_0)
$$ 
and therefore
\begin{equation}\label{eq:rho_0}
\frac{\lambda^{\tau(s)}\wt{\xi}_1(s)}{\wt{\xi}_2(s)}=1.
\end{equation}
Clearly, $\tau(s)$ is $C^\infty$ in $(0,s_0)$. 
Next, we will check that $\lambda^{\tau(s)}\wt{\eta}_1(s)=\wt{\eta}_2(s)$. Using that $H$ is a first integral and Lemma \ref{lemma:reparam_gamma_j} we have that
\[
H(\lambda^{\tau}\wt{\xi}_1(s), \mu^\tau\wt{\eta}_1(s) )
=
H(\wt{\xi}_1(s),\wt{\eta}_1(s))
=
s
=
H(\wt{\xi}_2(s),\wt{\eta}_2(s))
\]
from which we deduce that 
$$
\left(\frac{\lambda^{\tau}\wt{\xi}_1(s)}{\wt{\xi}_2(s)}\right)^{\log \mu^{-1}}=\left(\frac{\wt{\eta}_2(s)}{\mu^\tau\wt{\eta}_1(s)}\right)^{\log \lambda}.
$$
Taking $\tau=\tau(s)$, from \eqref{eq:rho_0}) we conclude
\begin{equation} \label{eq:mu_k}
\mu^{\tau(s)}\wt{\eta}_1(s)=\wt{\eta}_2(s),
\end{equation}
i.e. $\mathcal L ^\tau \left(\wt{\gamma}_1(s)\right)=\wt{\gamma}_2(s)$
as desired. Of course,
\begin{equation}\label{eq:limit}
\lim\limits_{s\to 0^+} \tau(s)= \frac{1}{\log \lambda}\log \lim\limits_{s\to 0^+} \frac{\wt{\xi}_2(s)}{\wt{\xi}_1(s)}
=\infty,
\end{equation}
since 
$\lim _{s\to 0^+} \wt{\xi}_2(s) = \xi_0>$ and
$\lim _{s\to 0^+} \wt{\xi}_1(s) = 0$. 
Using 
and the fact that $\tau(s)$ is a $C^\infty$ function in its domain we get from Bolzano's Theorem that there exist $k_0$ such that for every $k\geq k_0$ there exists $s_k$ such that  $t(s_k)=k$, and the lemma follows.
\end{proof}

The lemma above shows that the curves $\mathcal L^k\left(\wt{\gamma}_1(s)\right)$ and $\wt{\gamma}_2(s)$ intersect at the values $s=s_k$. Next lemma shows that these intersections, for $k$ large enough, are transversal.

\begin{lemma}\label{lemma:asym_angle}
Let $k$ be large enough. The following limits hold.
\begin{equation} \label{eq:limits}
\lim\limits_{s\to 0^+} \frac{\lambda^k\frac{d\wt{\xi}_1}{ds}(s)}{\frac{d\wt{\xi}_2}{ds}(s)}=\infty \qquad \text{and} \qquad \lim\limits_{s\to 0^+} \frac{\mu^k\frac{d\wt{\eta}_1}{ds}(s)}{\frac{d\wt{\eta}_2}{ds}(s)}=0. 
\end{equation}
In particular, for $k$ large enough, the curves $\mathcal L^k\left(\wt{\gamma}_1(s)\right)$ and $\wt{\gamma}_2(s)$ intersect transversally at the values $s=s_k$.
\end{lemma}

\begin{proof}
To prove the lemma we first make some computations. From the proof of Lemma \ref{lemma:reparam_gamma_j}, and equations \eqref{eq:G_1} and \eqref{eq:G_2}, we deduce the following asymptotic behaviours (as $s \to 0$)

\begin{equation}
\begin{array}{rlcl} \sigma_1(s) \! \! \!& = c_1^{-1} s^{1/(\ell_1\log \mu^{-1})}+\dots ,&\quad    \sigma_1^\prime(s) \! \!\!& = c_1^{-1}\frac{1}{\ell_1\log \mu^{-1}}s^{1/(\ell_1\log \mu^{-1})}s^{-1}+ \dots,  \\
\\
\sigma_2(s) \! \!\! & =   c_2^{-1} s^{1/(\ell_3\log \lambda)}+\dots,& \quad   \sigma_2^\prime(s) \! \! \! & = c_2^{-1}\frac{1}{\ell_3\log \lambda}s^{1/(\ell_3\log \lambda )}s^{-1}+\dots, 
 \end{array}
\end{equation}
where 
$$
c_1 = a_1^{1/\ell_1}\eta_0^{\log \lambda / (\ell_1 \log \mu^{-1})} \qquad \text{and} \qquad c_2 = a_3^{1/\ell_3} \xi_0 ^{\log \mu^{-1} / (\ell_3 \log \lambda)}
$$ 
are non-zero. Hence, using \eqref{eq:gammes_t}, we can compute the following expressions for some of the terms of the numerator and denominator in \eqref{eq:limits}. 
On the one hand, we have
\begin{equation*}
\begin{split}
\frac{d\wt{\xi}_1}{ds} (s)&=\frac{d(\xi_1\circ \sigma_1)}{ds}(s) = \sigma_1^{\ell_1-1}(s)\left[\ell_1(a_1+g_1(\sigma_1(s)))+\sigma_1(s)g_1^{\prime}(\sigma_1(s))\right] \sigma_1^{\prime}(s) \\ 
& =\ell_1a_1 \sigma_1^{\ell_1-1}(s) \sigma_1^{\prime}(s) +\dots =  \frac{a_1}{c_1^{\ell_1}\log\mu^{-1}}s^{1/\log\mu^{-1}}s^{-1} +\dots , \\
\frac{d \wt{\xi}_2}{ds}(s)&=\frac{d (\xi_2 \circ \sigma_2)}{ds}(s) = \sigma_2^{\ell_2-1}(s)\left[\ell_2(a_2+g_2(\sigma_2(s)))+\sigma_2(s)g_2^{\prime}(\sigma_2(s))\right]
\sigma_2^{\prime}(s)    
\\
&= \ell_2a_2 \sigma_2^{\ell_2-1}(s) \sigma_2^{\prime}(s) +\dots =  \frac{\ell_2 a_2}{\ell_3 c_2^{\ell_2}\log \lambda}s^{\ell_2/(\ell_3\log \lambda)}s^{-1} + \dots.
\end{split}
\end{equation*}
On the other hand, we have
\begin{equation*}
\begin{split}
\frac{d\wt{\eta}_1}{ds} (s)&=\frac{d(\eta_1 \circ \sigma_1) }{ds} (s) = \sigma^{\prime}_1(s) =c_1^{-1} \frac{1}{\ell_1\log \mu^{-1}} s^{1/\left(\ell_1\log\mu^{-1}\right)}s^{-1} + \dots, \\
\frac{d\wt{\eta}_2}{ds} (s)&=\frac{d(\eta_2\circ \sigma_2)}{ds}(s) =  \sigma_2^{\ell_3-1}(s)\left[\ell_3(a_3+g_3(\sigma_2(s)))+\sigma_2(s)g_3^{\prime}(\sigma_2(s))\right] \sigma_2^{\prime}(s) \\
&= \ell_3a_3 \sigma_2^{\ell_3-1}(s)\sigma_2^{\prime}(s) +\dots =\frac{a_3}{ c_2^{\ell_3}\log \lambda}s^{1/\log \lambda}s^{-1} + \dots.
\end{split}
\end{equation*}

From \eqref{eq:rho_0} and \eqref{eq:mu_k} evaluated at the values of $s= s_k$ corresponding to $\tau(s_k)=k$ we can also conclude that
$$
\lambda^k=\frac{\xi_2(\sigma_2(s_k))}{\xi_1(\sigma_1(s_k))}= \frac{\xi_0}{a_1}\sigma_1^{-\ell_1}(s_k) +\dots = \frac{\xi_0 c_1^{\ell_1}}{a_1}s^{-1/\log\mu^{-1}}+\dots 
 $$ 
 and
 $$
 \mu^k=\frac{\eta_2(\sigma_2(s_k))}{\eta_1(\sigma_1(s_k))} = \frac{a_3}{\eta_0}\sigma_2^{\ell_3}(s_k)+\dots  =\frac{a_3c_2^{-\ell_3}}{\eta_0}s_k^{1/ \log \lambda} +\dots .
 $$
All together allows us to compute the limits of the statement
\begin{equation*}
\begin{split}
&\lim\limits_{s\to 0^+} \lambda^k\frac{\frac{d\wt{\xi}_1}{ds}(s)}{\frac{d\wt{\xi}_2}{ds}(s)}=\lim\limits_{s\to 0^+} \frac{\xi_0 \ell_3\log \lambda}{\ell_2a_2c_2^{-1}r_2^{\ell_2-1}\log \mu^{-1}}s^{-\ell_2/(\ell_3\log \lambda)}=\infty \\
&\lim\limits_{s\to 0^+} \mu^k\frac{\frac{d\wt{\eta}_1}{ds}(s)}{\frac{d\wt{\eta}_2}{ds}(s)}=\lim\limits_{s\to 0^+} \frac{c_1^{-1}\log \lambda}{\eta_0 \ell_1\log \mu^{-1}}s^{1/(\ell_1\log \mu^{-1})}=0.
\end{split}
\end{equation*}
\end{proof}

\begin{proof}[End of the proof of Theorem B]
Let  $k\geq k_0$ satisfy the conditions of the previous lemmas.  We consider the sequence of points 
\begin{equation}\label{eq:s_k}
\Bigl\{\wt\gamma_j(s_k)=\left(\wt{\xi}_j(s_k),\wt{\eta}_j(s_k)\right)\Bigr\}_{k\geq k_0},\qquad  j=1,2.
\end{equation}
From Lemma \ref{lemma:asym_angle},
$ \mathcal L^k\left(\wt{\gamma}_1(s_k)\right)$ and $\wt{\gamma}_2(s_k)$ intersect transversally. It follows from \eqref{eq:gammes_t} and the lemmas above than 
\begin{equation}\label{eq:lim_gammes}
\wt{\gamma}_1(s_k)\to (0,\eta_0) \qquad \text{and} \qquad 
\wt{\gamma}_2(s_k) \to (\xi_0,0).
\end{equation} 
               We introduce the following notation 
\begin{equation}\label{eq:zks}
{\bf \wh  z}_k:= \left(\T^{-1}\circ \Theta^{-1} \circ \Psi^{-1}  \right) \left(\wt{\gamma}_1(s_k)\right) \in \wh{W}^u \qquad \text{ and} \qquad {\bf \wt z}_k:= \left(\T^{-1}\circ \Theta^{-1} \circ \Psi^{-1}  \right) \left(\wt{\gamma}_2(s_k)\right) \in \wt{W}^s, 
\end{equation}
and note that Lemma \ref{lemma:moving_Lrho} implies that
\begin{equation}\label{eq:zkzk_k}
{\bf \wt z}_k = T_d^{2k} \left(\wh {\bf z}_k\right),\qquad k\geq k_0.
\end{equation}
To conclude the proof of Theorem B we argue as follows. We know that  $\{{\bf \wh  z}_k\}_{k\geq k_0} \in W^u$ and  $\{{\bf \wt z}_k\}_{k\geq k_0} \in W^s$. Also the stable and  unstable manifolds are invariant sets for the map $T_d^2$. Finally, Lemma \ref{lemma:asym_angle}, definition \eqref{eq:zks} and equation \eqref{eq:zkzk_k} imply that  $\{{\bf \wh  z}_k\}_{k\geq k_0}$ and  $\{{\bf \wt z}_k\}_{k\geq k_0}$ correspond to transversal intersections of the stable and unstable manifolds of  $p_1$ accumulating to the points $q_s$ and $q_u$, respectively. See Figure \ref{fig:zks}.
\end{proof}

\begin{figure}[ht]
    \centering
     \includegraphics[width=0.85\textwidth]{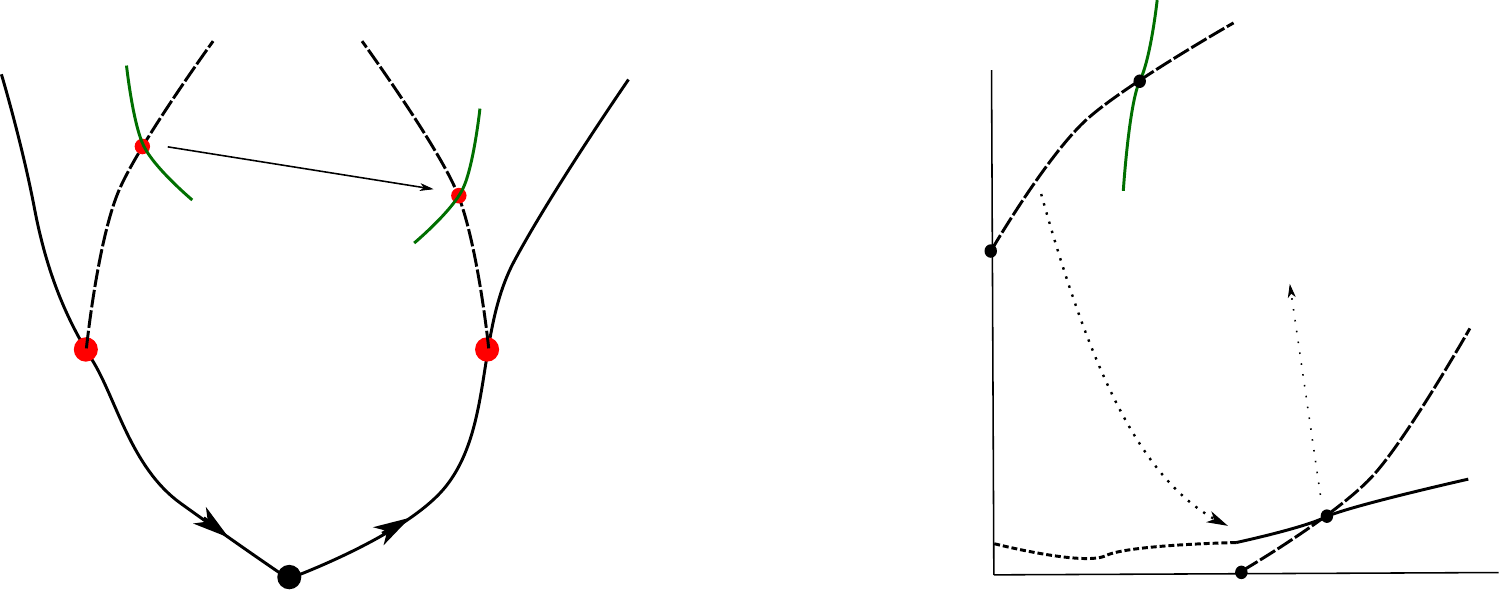}
      \put(-305,-5){\tiny $p_1$}
      \put(-370,60){\tiny $q_s$}
      \put(-245,60){\tiny $q_u$}
      \put(-300,140){\tiny $\wt{W}^s$}
      \put(-330,140){\tiny $\wh{W}^u$}
       \put(-359,124){\tiny $W^u$}
      \put(-253,115){\tiny $W^s$}
       \put(-353,110){\tiny $\wh{z}_k$}
      \put(-257,98){\tiny $\wt{z}_k$}
       \put(-310,96){\tiny $T_d^{2k}$}
       \put(-359,124){\tiny $W^u$}
       \put(-355,20){\tiny $W_{\rm loc}^s$}
        \put(-260,20){\tiny $W_{\rm loc}^u$}
       \put(-132,-3){\tiny $0$}
      \put(2,2){\tiny $\xi$}
       \put(-130,134){\tiny $\eta$}
       \put(-150,90){\tiny $(0,\eta_0)$}
       \put(-70,-4){\tiny $(\xi_0,0)$}
       \put(-112,50){\tiny $\mathcal{L}^k$}
       \put(-59,140){\tiny $\wt{\gamma}_1(s)$}
       \put(-115,130){\tiny $\wt{\gamma}_1(s_k)$}
       \put(-10,70){\tiny $\wt{\gamma}_2(s)$}
        \put(-83,82){\tiny $\mathcal{L}^k(\wt{\gamma}_1(s_k))=\wt{\gamma}_2(s)$}
         \put(-118,15){\tiny $\mathcal{L}^k(\wt{\gamma}_1(s))$}
                \caption{\small{Sketch of the proof of Theorem B with the points ${\bf \wh {z}}_k$ and  ${\bf \wt{z}}_k$ being transversal intersections of the stable and unstable manifolds of $p_1$ (for $T_d^2$).}}       
    \label{fig:zks}
    \end{figure}
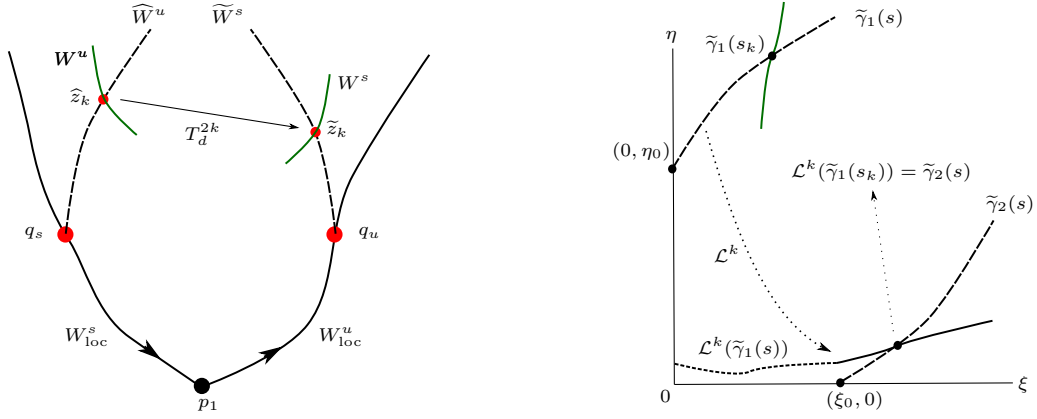

\section{Proof of Theorem C} \label{section:proof_C}

The proof of Theorem C is based on   Moser's version of Birkhoff-Smale Theorem, concretely we will apply Theorem 3.7 in \cite{Mos} in our setup. The key difficulty 
comes from  the fact that, in our case, $T_d^{-1}$ is not differentiable on the line $\{y=x\}$. Therefore, we need to make sure that the construction in  \cite{Mos} can be made so that we only have to deal with our map and its inverse in a domain that does not meet the line $\{y=x\}$. However, notice  that both the stable and the unstable manifolds of the two-cycle $\{p_0,p_1\}$ cross the line $y=x$. See Figure \ref{fig:intersection_global}.

\begin{remark}
	It follows from Lemma \ref{lemma:q1_q2_local} that $W^s$ as well as $W^u$ intersect the line $\{y=x\}$ at isolated points. In other words, finite length pieces of $W^s$ and $W^u$ only contain  finitely many intersections with $\{y=x\}$.
\end{remark}

Let $U$ be a neighbourhood of $p_1$ as in Lemma \ref{lemma:sternberg} where we can take local coordinates
for which $p_1$ is located at $(0,0)$ and 
the stable and the unstable manifolds of $(0,0)$ are the vertical and horizontal axes, respectively. Assume also that  $U\cap \{y=x\} =\emptyset$. 

In the following items, we summarize notation and facts of the constructions we have made in the previous section that will be important in the proof of Theorem C. See Figure \ref{fig:step2}.

\begin{itemize}
	\item[(a)]
	Let $q$ be the homoclinic point given in Theorem A and  $q_s= T_d^{2\alpha }(q)\in W^s_\loc\cap W^u$ and $q_u= T_d^{-2\beta }(q)\in W^u_\loc\cap W^s$, $\alpha, \beta \in \N$, be the points given in Lemma \ref{lemma:q1_q2_local}. 
	Let $\wh{W}^u \subset W^u$ and $\wt{W}^s\subset W^s$ introduced after the statement of Lemma \ref{lemma:q1_q2_local}. Then 
	$q_s\in W^s_\loc\cap \wh W^u$ and $q_u\in W^u_\loc\cap \wt W^s$. Moreover, taking $n_0= \alpha + \beta$ we have that 
	$$ 
	T_d^{2n_0}\left(q_u\right) = q_s.
	$$  
	\item [(b)] Let $\wh{\bf z}_k \in \wh W^u$ and  $ \wt{\bf z}_k\in \wt W^s$ be the points introduced in 
	\eqref{eq:zks}. We have that 
	$\wt{\bf z}_k$ and $\wh{\bf z}_k$ are transversal homoclinic points, 
	\[
	\lim\limits_{k\to \infty}  \hat{\bf z}_k =q_s \qquad \mbox{and} \qquad \lim\limits_{k\to \infty} \tilde{\bf z}_k =q_u.
	\]
	and 
	\[
	T_d^{2k}\left(\wh{\bf z}_k\right) = \wt{\bf z}_k, 
	\qquad 
	T_d^{2j}(\wh{\bf z}_k) \in U\quad \text{ for all } \quad j=1,\ldots ,k.
	\]
	
	\item[(c)] 
	Let
	$\wh{W}^s = T_d^{2n_0}(\wt{W}^s )$  and 
	$\wt{W}^u= T_d^{-2n_0} (\wh{W}^u)$.

	\item[(d)]  We consider the points 
	$$
	\wh{\bf w}_k:=T_d^{2n_0}(\wt{\bf z}_k) \in \wh W^s \cap  W^s_\loc \cap W^u  \qquad 
	\mbox{and} \qquad   \wt{\bf w}_k:=T_d^{-2n_0}(\wh{\bf z}_k) \in \wt W^u \cap W^u_\loc \cap W^s .
	$$
	As a consequence of the previous items we have
	$$
	T_d^{2n_0+2k+2n_0}( \wt{\bf w}_k)= \wh{\bf w}_k.
	$$
	\item [(e)]
	For any $m_u\geq 1$ and $m_s\geq 1$, if we write,
	$$
	{\bf w}_u:=T_d^{-2m_u}(\wt{\bf w}_k)\in W^u_\loc  \qquad 
	\mbox{and} \qquad  {\bf w}_s:=T_d^{2m_s}(\wh{\bf w}_k) \in W^s_\loc,
	$$
	
	we have   that
	\begin{equation}\label{eq:phi_outside}
		T_d^{2m_s+2n_0+2k+2n_0+2m_u}({\bf w}_u)={\bf w}_s \qquad \text{or} \qquad T_d^{-2m_s-2n_0-2k-2n_0-2m_u}({\bf w}_s)={\bf w}_u.
	\end{equation}
	In particular, considering $m_u$ and $m_s$ as large as necessary we know that the corresponding points ${\bf w}_u$ and ${\bf w}_s$ are as close as needed to the point $p_1$, and, by the $\lambda$-Lemma \cite{PALIS69} they are  transverse  homoclinic points  with tangent vectors close to the tangent vectors of the local manifolds. 
\end{itemize}

\begin{proof}[Proof of Theorem C]
Let $U_u\subset U$ be a neighbourhood of $q_u$. 
Assume it is sufficiently small so that
$U_s:=T_d^{2n_0}(U_u)$ is contained in $U$. Clearly, $U_s$  is a neighbourhood of $q_s$.  From items (b) and (d) there exists  $k_0>0$ such that $\wh{\bf z}_k,\wh{\bf w} _k \in U_s$ and $\wt{\bf z}_k,\wt{{\bf w}}_k \in U_u$ for all  $k\geq k_0$. 

Let  $k$ be  large enough and let $V$ be a small neighbourhood of $\bf w_u$ (suitable size will be decided later on). Then $T_d^{2m_u}(V)$ is a neighbourhood of $\wt{{\bf w}}_k \subset U_u \subset U$ and,  if we denote 
$$
m_0:=m_u+n_0+k_0+n_0+m_s,
$$ 
then $T_d^{2m_0}(V)$  is a neighbourhood of ${\bf w}_s$.

Let $R\subset V$ be a pseudo-rectangle in \textit{the first quadrant} attached to ${\bf w}_u$  whose boundaries are given by pieces of $W^s$ and $W^u$ and the others are just straight lines (parallel to the tangent lines of $W^s$ and $W^u$ at ${\bf w}_u$).

Define
\begin{equation}\label{eq:R}
	\wt{R}:=T_d^{2m_0}(R).
\end{equation}
If we iterate $\wt R $ by $T_d^2$, while staying in $U$ where the dynamics is $C^\infty$ conjugate to the one of the linearization at $p_0$, we eventually meet $R$.

Following Moser we introduce the {\it transversal map} $ \wt{\phi}$. Given $\xi\in R$ we consider a number of iterates bigger that $m_0$ of $T^2_d$. By construction $T^{2m_0}_d(\xi) \in \wt R$. Next we consider   $k=k(\xi)>m_0$ to be the smallest integer such that $T_d^{2k}(T^{2m_0}_d(\xi)) \in R$ and $T_d^{2j}(T^{2m_0}_d(\xi)) \in U$, $1\le j\le k$, if it exists. We denote by $\D$ the set of $\xi\in R$ such that $k(\xi)$ exists and we define   
$
\wt{\phi}: \mathcal D \subset  R  \to R
$ 
by 
\begin{equation}\label{eq:phi_tilde}
\wt{\phi} (\xi) = T_d^{2m_0+ 2k(\xi)}(\xi), \qquad \xi\in \D.
\end{equation}

To apply Moser's Theorem we should check that the restriction of $T_d$ to $\bigcup_{k=1}^{2m_0} T_d^k(\D)$ is a  $\mathcal C^\infty$ diffeomorphism or equivalently that $\bigcup_{k=1}^{2m_0} T_d^k(\D) \cap \{y=x\} =\emptyset$.

That is, we should prove that, by choosing $R$ small enough the points travelling from $R$ to itself would not meet the line $\{y=x\}$, where $T_d^{-1}$ is not smooth.

Since $\xi \in \mathcal D \subset R\subset V$ and $V$ is a small neighbourhood of ${\bf w}_u$ the iterates of $\xi \in \mathcal D$ will travel 
following the orbit of ${\bf w}_u$  
$$
{\bf w}_u \xmapsto{T_d^{2m_u}}  \wt{{\bf w}}_k \xmapsto{T_d^{2n_0}}  {\wh {\bf z}}_k \xmapsto{T_d^{2k}}  {\wt {\bf z}}_k \xmapsto{T_d^{2n_0}}  \wh{{\bf w}}_k \xmapsto{T_d^{2m_s}}  {\bf w}_s 
$$
until they arrive to $\wt R$. Then, by item (b),  from $\wt R$ to $R$ the iterates will stay in $U$.
Hence, the only iterates of points $q\in \mathcal D \subset R$ which might fall in the line $\{y=x\}$ are the $2n_0$ iterates needed  to go from $\wt{\bf w}_k$ to $\wh{\bf z}_k$ and the  $2n_0$ ones from $\wt{{\bf z}}_k$ to $\wh{{\bf w}}_k$. To finish the argument, we distinguish two cases.

\vglue 0.2truecm
\noindent{\it Case 1.} The finite set $\{T_d^{j}(q_u)\mid\, 0 \leq j \leq  2n_0\}$ does not intersect $\{y=x\}$. By continuity there exists a sufficiently small open neighbourhood $U_{u}$ of $q_u$ such that the open set 
$$
\mathcal U_{u}:= \bigcup_{j=0}^{2n_0} T_d^{j}(U_{u})
$$ 
does not intersect $\{y=x\}$ either. 
Of course, by items (b) and (d) there are infinitely many points of the sequences 
$\{\wt{\bf w}_k\}, \{\wt{\bf z}_k\}$ belonging to $\mathcal U_{u}$. 

Choose $V$ (the neighbourhood of ${\bf w}_u$ above) small enough and $k$ large enough such that $T_d^{2m_u}(V) \subset U_u$ and such that $T_d^{2j}(V)$   belong to $U$, for $j=0,\ldots ,m_u$. Choose $R\subset V$ and define $\wt{R}$ as in \eqref{eq:R}. By construction,  the map $\wt{\phi}$ is well defined and $\mathcal C^\infty$, it has an inverse $\wt{\phi}^{-1}: R\cap \wt{\phi}(R) $ which is also $\mathcal C^\infty$  since no iterates of $T_d$ or $T_d^{-1}$ in the definition of $\wt{\phi}$ (see \eqref{eq:phi_tilde}) intersect $\{y=x\}$.

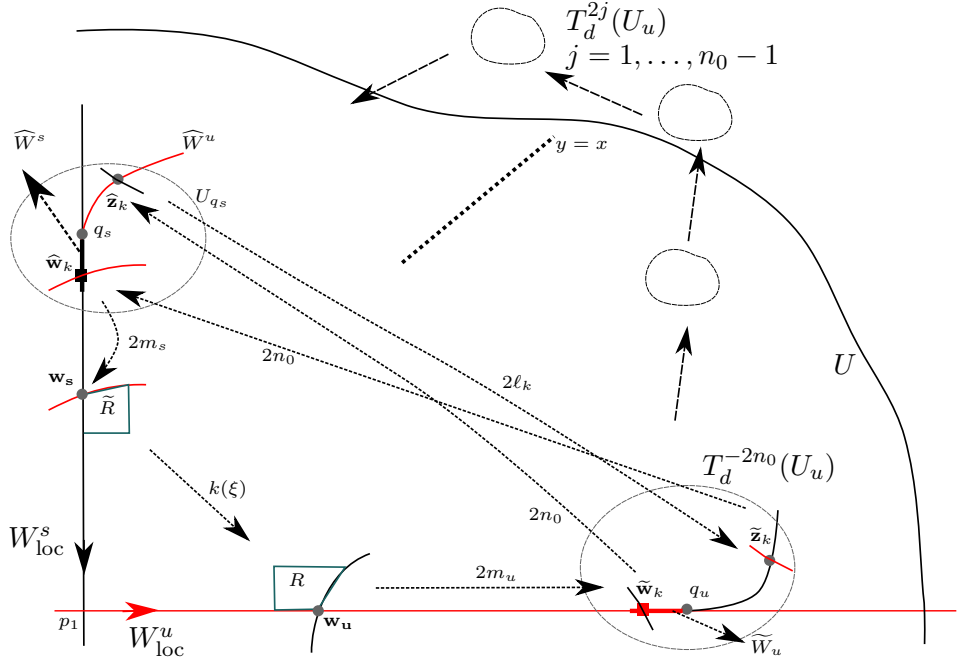
\begin{figure}[ht]
	\centering
	\includegraphics[width=0.8\textwidth]{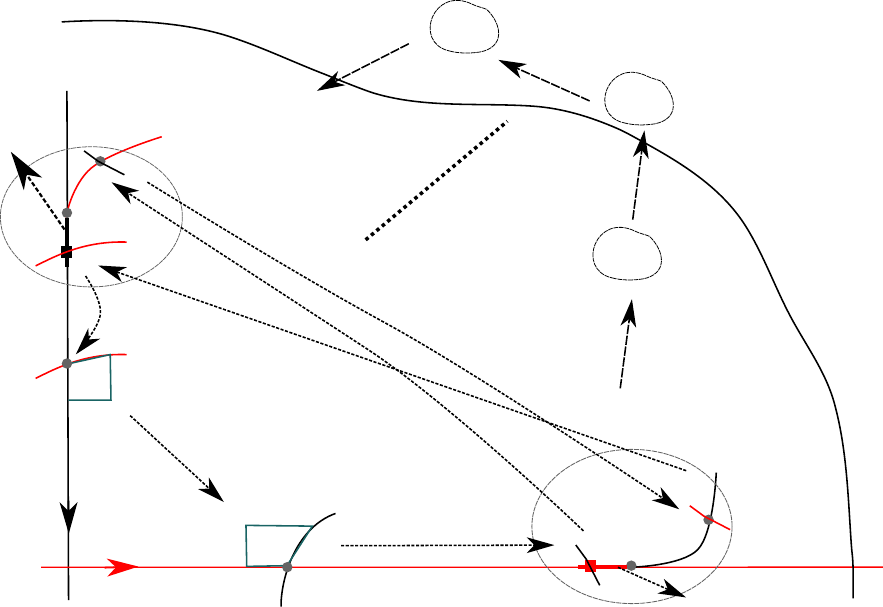}
	 \put(-150,190){\tiny $y=x$}
	 \put(-290,190){\tiny $\wh{W}^u$}
	  \put(-318,168){\tiny ${\wh{\bf z}}_k$}
	   \put(-341,145){\tiny ${ \wh{\bf w}}_k$} 
	     \put(-353,190){\tiny $\wh{W}^s$} 
	   \put(-323,158){\tiny $q_s$}
	   \put(-285,170){\tiny $U_{q_s}$}
	  \put(-340,102){\tiny ${\bf w_s}$}
	  \put(-321,90){\tiny $\wt{R}$}
	  \put(-250,25){\tiny $R$}
	  \put(-336,10){\tiny $p_1$}
	   \put(-46,105){$U$} 
	   \put(-146,235){$\tiny{T^{2j}_d(U_{u})}$}
	   \put(-146,222){$\tiny j=1,\ldots,  n_0-1$}
           \put(-355,40){$W_{\rm loc}^s$}
           \put(-310,2){$W_{\rm loc}^u$}
           \put(-237,10){\tiny $\bf {w}_u$}
           \put(-120,22){\tiny ${ \wt{\bf w}}_k$}
           \put(-77,0){\tiny $\wt{W}_u$}
            \put(-100,22){\tiny $q_u$}
            \put(-95,67){$\tiny{T_d^{-2n_0}(U_{u})}$}
            \put(-77,42){\tiny ${ \wt{\bf z}}_k$}
            \put(-180,28){\tiny $2m_u$}
            \put(-160,50){\tiny $2n_0$}
            \put(-260,110){\tiny $2n_0$}
            \put(-170,100){\tiny $2\ell_k$}
             \put(-310,115){\tiny $2m_s$}
             \put(-280,60){\tiny $k(\xi)$}
	\caption{\small{The illustration of all items  (a)-(f).}}       
	\label{fig:step2}
\end{figure}

\noindent{\it Case 2.} 
The set $\{T_d^{j}(q_u)\mid\, 1\le j \le 2n_0\}$ intersects  $\{y=x\}$.
Let $\{T_d^{\ell_j}(q_u)\mid\, j=1,\ldots, r\}$,
for $0<\ell_1  <\ldots <\ell_r<2n_0$, for some $r\geq 1$, be the intersection. 
We recall that the 
stable and unstable manifolds of the point $p_1$ have discrete intersection with $\{y=x\}$.
Again, items (b) and (d) imply that we can choose $k_0$ large enough such that there exist two small open neighbourhoods $U_{s}$ of $q_s$ and  $U_{u}$ of $q_u$ such that for all $k\geq k_0$ we have that 
\begin{equation*}
	\begin{split}
		\{T_d^{j}(U_{u}\setminus \{q_u\})\mid \, j=0,\ldots ,2n_0\} \cap \{y=x\} =\emptyset \qquad \mbox{and} \qquad \wt{\bf w}_k, \wt{\bf z}_k \in U_{u}.
	\end{split}
\end{equation*}
We are in the same situation as in the previous case. Therefore, choosing $V$ (the neighbourhood of ${\bf w}_u$ above) small enough and  $k$  large enough we obtain the regularity claim for $\wt{\phi}$
and $\wt{\phi}^{-1}$.

Now, Theorem 3.7 in \cite {Mos} implies that there is a Cantor set $\I$ contained  in $R$ and a homeomorphism from $\I$ to the space of sequences of $N$ symbols ($2\le N\le \infty$) which conjugates $\wt \phi $ with the Bernoulli shift and, as a consequence,
there is a dense set $\PP$ of periodic orbits of 
$\wt \phi $,  and therefore of $T_d^2$ and $T_d$ in $\I$.

Moreover, Theorem 3.8 in  \cite{Mos} implies that there is a dense subset $\HH$ of homoclinic points to $p_1$ in $\I$.  We recall from Theorem A (b) of \cite{FGJ24} that  $W^s_{p_1} \subset \partial \Omega$. Finally, since 
$\HH \subset W^s_{p_1}$, 
$$
\PP\subset \ol \PP = \I = \ol \HH \subset \ol{\partial \Omega} = {\partial \Omega},
$$
that is, the boundary of $\Omega$ has infinitely many periodic orbits with arbitrary high period.

\end{proof}

\bibliographystyle{alpha}
\bibliography{biblio}

\end{document}